\documentclass[10pt]{amsart}

\usepackage[a4paper, left=25mm, right=25mm, top=25mm, bottom=25mm]{geometry}
\usepackage{lipsum}
\usepackage{amsfonts}
\usepackage{graphicx}
\usepackage{epstopdf}
\usepackage{amssymb}
\usepackage{amsmath}
\usepackage{mathrsfs}
\usepackage{stmaryrd}
\usepackage{chemarrow}
\usepackage{enumerate}
\usepackage{graphicx}
\usepackage{subfigure}  
\usepackage{hyperref}
\usepackage[hyperpageref]{backref}
\usepackage[all]{xy}
\usepackage{tikz}
\usetikzlibrary{arrows}
\usepackage{comment,enumerate,multicol,xspace}
\usepackage{booktabs}
\usepackage{array}
\usepackage{listings}
\usepackage[numbered]{mcode}
\usepackage{algorithmicx}
\usepackage{algpseudocode}
\usepackage{longtable} 
\usepackage{algorithm}

\newcommand{\TT}{\texttt{T}}
\newcommand{\dx}{\,{\rm d}x}
\newcommand{\dd}{\,{\rm d}}
\newcommand{\bs}{\boldsymbol}

\newcommand{\tr}{\operatorname{tr}}

\newcommand{\dist}{\operatorname{dist}}
\newcommand{\sym}{\operatorname{sym}}

\newcommand{\sign}{\operatorname{sign}}
\newcommand{\Oplus}{\ensuremath{\vcenter{\hbox{\scalebox{1.5}{$\oplus$}}}}}

\newcommand{\ttt}[1]{\textnormal{\texttt{#1}}}

\newcommand{\step}[1]{\noindent\raisebox{1.5pt}[10pt][0pt]{\tiny\framebox{$#1$}}\xspace}

\newtheorem{theorem}{Theorem}[section]
\newtheorem{lemma}[theorem]{Lemma}
\newtheorem{corollary}[theorem]{Corollary}
\newtheorem{remark}[theorem]{Remark}

\begin{document}
\title[Basis Construction of Smooth Finite Elements]{Implementation and Basis Construction for Smooth Finite Element Spaces}
\author{Chunyu Chen}%
 \address{School of Mathematics and Computational Science, Xiangtan University,
    National Center for Applied Mathematics in Hunan, 
    Hunan Key Laboratory for Computation and Simulation in Science and Engineering, 
    Xiangtan 411105, China}%
 \email{cbtxs@smail.xtu.edu.cn}%
\author{Long Chen}%
 \address{Department of Mathematics, University of California at Irvine, Irvine, CA 92697, USA}%
 \email{chenlong@math.uci.edu}%
\author{Tingyi Gao}%
 \address{School of Mathematics and Computational Science, Xiangtan University,
    National Center for Applied Mathematics in Hunan, 
    Hunan Key Laboratory for Computation and Simulation in Science and Engineering, 
    Xiangtan 411105, China}%
 \email{202431510119@smail.xtu.edu.cn}%
 \author{Xuehai Huang}%
 \address{School of Mathematics, Shanghai University of Finance and Economics, Shanghai 200433, China}%
 \email{huang.xuehai@sufe.edu.cn}%
\author{Huayi Wei}%
 \address{School of Mathematics and Computational Science, Xiangtan University,
    National Center for Applied Mathematics in Hunan, 
    Hunan Key Laboratory for Computation and Simulation in Science and Engineering, 
    Xiangtan 411105, China}%
\email{weihuayi@xtu.edu.cn}%

\thanks{The first, third and fifth authors were supported by the
    National Key R\&D Program of China (2024YFA1012600) and the National Natural Science Foundation of China (NSFC) (Grant No. 12371410, 12261131501) .The second author was supported by NSF (Grant No. DMS-2309777, DMS-2309785). The fourth author was supported by NSFC (Grant No. 12171300).}

\makeatletter
\@namedef{subjclassname@2020}{\textup{2020} Mathematics Subject Classification}
\makeatother
\subjclass[2020]{
65N30;   
65N12;   
31B30;   
}


 \begin{abstract}
The construction of $C^m$ conforming finite elements on simplicial meshes has recently advanced through the groundbreaking work of Hu, Lin, and Wu (Found. Comput. Math. 24, 2024). Their framework characterizes smoothness via moments of normal derivatives over subsimplices, leading to explicit degrees of freedom and unisolvence, unifying earlier constructions. However, the absence of explicit basis functions has left these spaces largely inaccessible for practical computation. In parallel, multivariate spline theory (Chui and Lai, J. Approx. Theory 60, 1990) enforces $C^m$ smoothness through linear constraints on Bernstein--B\'{e}zier coefficients, but stable, locally supported bases remain elusive beyond low dimensions. Building on the geometric decomposition of the simplicial lattice proposed by Chen and Huang (Math. Comp. 93, 2024), this work develops an explicit, computable framework for smooth finite elements. The degrees of freedom defined by moments of normal derivatives are modified to align with the dual basis of the Bernstein polynomials, yielding structured local bases on each simplex. Explicit basis construction is essential not merely for completeness, but for enabling efficient matrix assembly, global continuity, and scalable solution of high-order elliptic partial differential equations. This development closes the gap between theoretical existence and practical realization, making smooth finite element methods accessible to broad computational applications.
\end{abstract}

\keywords{Smooth finite element, Polyharmonic equation, Bernstein basis function}

\maketitle


\section{Introduction}
The construction of smooth finite element spaces on simplicial meshes is fundamental in numerical analysis, with broad applications in numerical methods for high-order partial differential equations (PDEs), isogeometric analysis, and geometric processing, among others.

Recently, Hu, Lin, and Wu~\cite{HuLinWu2024} achieved a major breakthrough by constructing $C^m$-conforming finite element spaces in arbitrary dimensions. Their framework characterizes smoothness via moments of normal derivatives over subsimplices, leading to explicit degrees of freedom (DoFs) and unisolvence, unifying earlier constructions in two~\cite{bramble_triangular_1970,argyris_tuba_1968,zenisek_interpolation_1970}, three~\cite{vzenivsek1974tetrahedral,Lai;Schumaker:2007Trivariate,zhang_family_2009}, and four dimensions~\cite{Zhang2016a}.

Later, Chen and Huang~\cite{chen_geometric_2021,Chen;Huang:2022FEMcomplex3D} introduced a geometric perspective, decomposing the simplicial lattice based on a distance structure, and emphasizing the underlying geometric organization, in contrast to the more combinatorial framework of Hu, Lin, and Wu~\cite{HuLinWu2024}.

Despite these advances, a critical gap remains: neither the Hu--Lin--Wu nor the Chen--Huang constructions provides explicit, computable bases. Without such bases, the practical implementation of smooth finite elements remains difficult. In computational mathematics, algorithmic realizability is often as impactful as theoretical construction. Explicit basis functions are essential not merely for completeness, but for enabling efficient matrix assembly, global continuity, and scalable application to numerical methods for high-order PDEs.

Meanwhile, multivariate spline theory~\cite{Chui;Lai:1990Multivariate,LaiSchumaker2007} typically enforces $C^m$ continuity through linear constraints on Bernstein--B\'{e}zier coefficients across element boundary. Although this interpolation-oriented approach is conceptually intuitive, constructing stable, locally supported bases remains challenging, particularly in three and higher dimensions~\cite{Davydov2001}.

This work develops an explicit basis construction and implementation strategy for $C^m$-conforming finite element spaces on simplicial meshes in arbitrary dimension. The degrees of freedom, originally defined by moments of normal derivatives over sub-simplices~\cite{HuLinWu2024}, are modified by replacing the integral moments with evaluations against the dual basis of the Bernstein polynomials. Based on these modified DoFs, basis functions are systematically constructed by solving a lower-triangular DoF--basis system with structured, explicitly computable entries. This structure enables efficient inversion, local basis construction, and practical global assembly of smooth finite element spaces. Numerical validation is provided through interpolation of smooth functions and the solution of high-order elliptic problems in two and three dimensions.

The proposed framework paves the way for the practical deployment of smooth
finite elements in high-dimensional and high-smoothness regimes. The resulting
$C^m$ elements are implemented on the FEALPy platform~\cite{fealpy}, supporting multiple backends such as NumPy~\cite{harris2020array}, PyTorch~\cite{paszke2019pytorch}, and JAX~\cite{jax2018github}, optimizing both accessibility and computational performance.

\subsection*{\textbf{Technical Outline}}

The construction begins with the decomposition of the \emph{simplicial lattice} $
\mathbb{T}^d_k := \left\{ \alpha \in \mathbb{N}^{0:d} : |\alpha| = k \right\}$ presented in~\cite{chen_geometric_2021,Chen;Huang:2022FEMcomplex3D}. Given an integer $m \geq 0$ and a sequence $\bs{r} = (r_0, \ldots, r_d)$ satisfying
\[
r_d = 0, \quad r_{d-1} = m, \quad r_\ell \geq 2r_{\ell+1} \quad \text{for} \quad \ell = d-2, \ldots, 0,
\]
and assuming $k \geq 2r_0 + 1 \geq 2^d m + 1$, the lattice $\mathbb{T}_k^d(T)$ embedded into a simplex $T$ admits the direct sum decomposition
\[
\begin{aligned}
\mathbb{T}_k^d(T) &= \Oplus_{\ell=0}^d \Oplus_{f \in \Delta_\ell(T)} S_\ell(f), \\
\text{where} \quad S_\ell(f) &= D(f, r_\ell) \setminus \left( \bigcup_{i=0}^{\ell-1} \bigcup_{e \in \Delta_i(f)} D(e, r_i) \right),
\end{aligned}
\]
where $D(f, r_\ell) = \left\{ \alpha \in \mathbb{T}_k^d : \dist(\alpha, f) \leq r_{\ell} \right\}$ is defined by a distance defined on $\mathbb{T}^d_k$. 

Next, we give explicit formulas for derivatives of Bernstein polynomials. Let $B^{\beta} = \frac{k!}{\beta!}\lambda^{\beta}$ denote the Bernstein polynomial associated to a lattice point $\beta \in \mathbb{T}_k^d$, where $\lambda = (\lambda_0, \lambda_1, \ldots, \lambda_d)$ are the barycentric coordinates. The $r$-th order derivative of $B^{\beta}$ is given by
\begin{equation}\label{intro:bern_grad}
\nabla^r B^{\beta} = \sum_{\substack{\alpha \in \mathbb{T}_r^d \\ \alpha \leq \beta}} \frac{r!k!}{(k-r)!\alpha!} \operatorname{sym}((\nabla\lambda)^{\otimes \alpha}) B^{\beta - \alpha}.
\end{equation}

Let $\{\texttt{b}^{\alpha}\}$ denote the dual basis to the Bernstein basis, defined by
\[
\texttt{b}^{\alpha}(B^{\beta}) := \langle \texttt{b}^{\alpha}, B^{\beta} \rangle = \delta_{\alpha,\beta}, \quad \alpha, \beta \in \mathbb T_k^d.
\]
The modified degrees of freedom take the form
\begin{equation}\label{intro:bDoF}
\left\langle \texttt{b}^{\alpha_f}, \frac{\partial^{|\alpha_{f^*}|}}{\partial n_f^{\alpha_{f^*}}} u \bigg|_f \right\rangle, \quad \alpha = \alpha_f + \alpha_{f^*} \in S_\ell(f).
\end{equation}

By selecting local normal bases $\{n_f\}$ for the normal planes associated with each sub-simplex $f$ and applying~\eqref{intro:bern_grad}, the entries of the DoF--basis matrix can be explicitly computed as
\[
D_{\alpha,\beta} = \left\langle \texttt{b}^{\alpha_f}, \nabla^{|\alpha_{f^*}|} B^{\beta} : n_f^{\alpha_{f^*}} \right\rangle, \quad \alpha, \beta \in \mathbb T_k^d.
\]
The resulting matrix $(D_{\alpha,\beta})$ is block lower triangular, with each diagonal block being a positively scaled identity, allowing for efficient local basis construction via direct inversion of $(D_{\alpha,\beta})$.

Let $\mathcal T_h$ be a shape regular triangulation.  A reference lattice decomposition is introduced based on $\mathcal T_h$
\[
\begin{aligned}
\mathbb S_{k,\bs r}^d:=& \Oplus_{\ell=0}^d\Oplus_{f\in \Delta_{\ell}(\mathcal T_h)}\hat{S}_{\ell}([f]), \quad \text{ where }\\
\hat{S}_{\ell}([f]):= &\{(\alpha_f, \gamma) \mid \alpha_f \in R_f(S_{\ell}(f)), \gamma \in \mathbb{T}_{k - |\alpha_{f}|}^{d-\ell - 1} \}.
\end{aligned}
\]
To enforce global $C^m$ continuity, fixed normal bases $\{\bs{N}_f^1, \ldots, \bs{N}_f^{d-\ell}\}$ are selected for the normal planes of each sub-simplex $f$, depending only on ascending ordered $[f]$ not on specific elements containing it.

Using the reference lattice decomposition and a global normal basis, the global degrees of freedom
\begin{equation}\label{intro:globalDoF}
\langle \texttt{b}^{\alpha_f}, \frac{\partial^{|\gamma|}}{\partial
    N_f^{\gamma}}u \mid_f \rangle, \quad \alpha = (\alpha_f, \gamma) \in
    \hat{S}_{\ell}([f]),\ f\in \Delta_{\ell}(\mathcal{T}_h),\ \ell = 0, 1, \ldots, d,
\end{equation}
are thus independent of the element containing $f$. The local and global normal derivatives are related through a change of basis, allowing transformation between local DoFs \eqref{intro:bDoF} and global DoFs \eqref{intro:globalDoF}.

\subsection*{\textbf{Organization of the Paper}}

The remainder of the paper is organized as follows. Section~\ref{sec:prelim} introduces basic notation, the simplicial lattice, and the geometric decomposition. Section~\ref{sec:B} introduce Bernstein polynomial and the formulas of its derivatives. Section~\ref{sec:local_frame} develops the construction of local basis functions, including the modification of degrees of freedom and the explicit computation of the DoF--basis matrix. Section~\ref{sec:global} addresses the enforcement of global $C^m$ continuity and describes the assembly of the global finite element space. Section~\ref{sec:numerics} presents numerical experiments validating the approximation properties of the constructed spaces through interpolation and the solution of high-order elliptic PDEs. As an example, we work out bases for $C^1$ element in two dimensions in Appendix~\ref{appdx:C12d}.

\section{Geometric Decomposition of the Simplicial Lattice}\label{sec:prelim}

In this section, we review the geometric decomposition of the simplicial lattice introduced in~\cite{chen_geometric_2021} and~\cite[Appendix~A]{Chen;Huang:2022FEMcomplex3D}. We refine the notation in
~\cite{chen_geometric_2021,Chen;Huang:2022FEMcomplex3D} by introducing abstract simplex and simplicial complex. 

\subsection{Geometric and Abstract Simplices}

Let $T \subset \mathbb{R}^{d}$ be a geometric $d$-simplex with vertices $\{\texttt{v}_0, \ldots, \texttt{v}_d\}$, defined as
\[
T = \operatorname{Convex}(\texttt{v}_0, \ldots, \texttt{v}_d) := \left\{ \sum_{i=0}^d \lambda_i \texttt{v}_i : \lambda_i \geq 0,\ \sum_{i=0}^d \lambda_i = 1 \right\},
\]
where $\lambda = (\lambda_0, \ldots, \lambda_d)$ are the barycentric coordinates. If the vertices $\texttt{v}_0, \ldots, \texttt{v}_d$ are affinely independent, then $T$ has nonzero $d$-dimensional volume.

It is convenient to introduce a standard reference simplex $\hat{T} \subset \mathbb{R}^d$, defined by the convex hull of the points $\{\bs{0}, \bs{e}_1, \ldots, \bs{e}_d\}$, where $\bs{e}_i$ denotes the $i$-th coordinate unit vector. Any geometric $d$-simplex $T$ can then be represented as the image of $\hat{T}$ under an affine transformation. In traditional finite element methods~\cite{Ciarlet1978}, calculations are often performed on the reference simplex $\hat T$ and transferred to the physical element via such affine mappings.

An abstract $d$-simplex $\TT$ is a finite set of cardinality $d+1$. Analogous to the reference simplex, the standard (combinatorial) $d$-simplex is the abstract simplex $\texttt{S}_d := \{0, 1, \ldots, d\}$. Any abstract $d$-simplex $\TT = \{\TT(0), \ldots, \TT(d)\}$ is combinatorially isomorphic to $\texttt{S}_d$ via the indexing map $i \mapsto \TT(i)$ for $i = 0, \ldots, d$. For example, $\TT = \{12, 10, 25\}$ is an abstract simplex, where $\TT(i)$ represents the global index of the $i$-th vertex. Therefore, $\TT$ can also be thought of as a local-to-global index mapping.

In practice, it is sufficient to work with the standard abstract simplex $\texttt{S}_d$, and generalize results to an arbitrary abstract simplex $\TT$ using the indexing map. Each geometric $d$-simplex determines an abstract $d$-simplex through its vertex set. Conversely, any abstract $d$-simplex can be realized geometrically by assigning its elements to distinct points in $\mathbb{R}^n$ for some $n \geq d$. Given a geometric simplex $T$ with vertices $\{\texttt{v}_0, \ldots, \texttt{v}_d\}$ and an abstract simplex $\TT = \{\TT(0), \ldots, \TT(d)\}$, we say that $T$ is a geometric realization of $\TT$, denoted by $\TT(T)$, via the correspondence $\TT(i) \mapsto \texttt{v}_i$.

Note that a single abstract $d$-simplex may have multiple geometric realizations. For instance, $\texttt{S}_d(T_1) \neq \texttt{S}_d(T_2)$ for two different geometric simplices $T_1$ and $T_2$. However, the combinatorial structure derived from $\texttt{S}_d$ remains invariant and can be transferred to any abstract simplex $\TT$ through the index map.

\subsection{The Simplicial Lattice}

For integers $l \leq m$, let $\alpha \in \mathbb{N}^{l:m}$ denote a multi-index $\alpha = (\alpha_l, \ldots, \alpha_m)$ with nonnegative integer entries. The total degree is defined by $|\alpha| := \sum_{i=l}^m \alpha_i$. For $\alpha, \beta \in \mathbb{N}^{l:m}$, we write $\alpha \geq \beta$ if $\alpha_i \geq \beta_i$ for all $i = l, \ldots, m$, and $\alpha \geq c \in \mathbb{R}$ if $\alpha_i \geq c$ for all $i$.

The \emph{simplicial lattice} of degree $k$ in dimension $d$ is defined as
\[
\mathbb{T}^d_k := \left\{ \alpha \in \mathbb{N}^{0:d} : |\alpha| = k \right\},
\]
whose elements are referred to as \emph{lattice points}.

Given a geometric $d$-simplex $T$ with vertices $\texttt{v}_0, \ldots, \texttt{v}_d$, the lattice $\mathbb{T}^d_k$ can be embedded into $T$ by interpreting each multi-index $\alpha \in \mathbb{T}^d_k$ as barycentric coordinates scaled by $1/k$. Specifically, define the mapping
\[
x: \mathbb{T}^d_k \to T, \quad x_\alpha := \sum_{i=0}^d \lambda_i(\alpha) \, \texttt{v}_i, \quad \lambda_i(\alpha) := \frac{\alpha_i}{k}.
\]
In this way, each $\alpha \in \mathbb{T}^d_k$ corresponds to the point $x_\alpha \in T$ whose barycentric coordinates are $\alpha / k$. The image of this embedding is called the \emph{geometric realization} (or \emph{embedding}) of $\mathbb{T}^d_k$, and is denoted by $\mathbb{T}^d_k(T)$; see Fig.~\ref{fig:lattice2D} for an illustration of $\mathbb{T}^2_8(T)$. This structure was introduced as the \emph{principal lattice} in~\cite{nicolaides1973class}.

While $\mathbb{T}^d_k$ is a purely combinatorial object, its geometric realization $\mathbb{T}^d_k(T)$ permits the application of geometric and analytic operations within the simplex $T$. For any subset $D \subseteq T$, define
\[
\mathbb{T}^d_k(D) := \{ \alpha \in \mathbb{T}^d_k : x_{\alpha} \in D \}
\]
as the set of lattice points whose geometric images lie within $D$. In particular, $\mathbb{T}^d_k(\partial T)$ denotes the subset of lattice points that lie on the boundary $\partial T$.

\subsection{Sub-simplices and Sub-simplicial Lattices}

Let $\TT$ be an abstract $d$-simplex. For an integer $0 \leq \ell \leq d$, any subset of $\TT$ with cardinality $\ell + 1$ is called an $\ell$-dimensional sub-simplex. The set of all such $\ell$-dimensional sub-simplices is denoted by $\Delta_{\ell}(\TT)$. The full collection of sub-simplices is given by the disjoint union
\[
\Delta(\TT) := \Oplus_{\ell=0}^d \Delta_{\ell}(\TT).
\]
The cardinalities satisfy $|\Delta_{\ell}(\TT)| = \binom{d+1}{\ell+1}$ and $|\Delta(\TT)| = 2^{d+1} - 1$.

For $f \in \Delta_{\ell}(\texttt{S}_d)$, we define the relabeling map
\[
f(\TT) := \{\TT(f(0)), \ldots, \TT(f(\ell))\} \in \Delta_{\ell}(\TT),
\]
which induces an isomorphism between $\Delta_{\ell}(\texttt{S}_d)$ and $\Delta_{\ell}(\TT)$. Vice verse, given an $f\in \Delta_{\ell}(\TT)$, we use notation
$$
f(\texttt{S}_d)\in \Delta_{\ell}(\texttt{S}_d), \text{ so that } (f(\texttt{S}_d))(\TT) = f.
$$
If we treat $\TT$ as the local-to-global index mapping, $f(\TT) = \TT \circ f(\texttt{S}_d)$ and $f(\texttt{S}_d) = \TT^{-1}\circ f(\TT)$. 
For example, for $\TT = \{ 2, 10, 25, 78\}$ and $f(\TT) = \{10, 25\}$, then $f(\texttt{S}_3) = \{1, 2\}$. And for $f(\texttt{S}_3) = \{0, 1\}\in \Delta_{1}(\texttt{S}_3)$, $f(\TT) = \{2, 10\}$. 

Given a geometric realization $T$ of $\TT$, each sub-simplex $f \in \Delta_\ell(\texttt{S}_d)$ (for $0 \leq \ell \leq d$) induces a geometric $\ell$-simplex defined by
\[
f(T) := \operatorname{Convex}(\texttt{v}_{f(0)}, \ldots, \texttt{v}_{f(\ell)}).
\]
Accordingly, the set of geometric $\ell$-simplices is denoted as
\[
\Delta_{\ell}(T) := \{f(T) : f \in \Delta_\ell(\texttt{S}_d)\}.
\]
For simplicity, we may use a single notation $f$ to refer to the abstract sub-simplex $f \in \Delta_\ell(\texttt{S}_d)$, its relabeled version $f(\TT) \in \Delta_\ell(\TT)$, and its geometric realization $f(T)$, unless clarification is required.

For $f \in \Delta_\ell(\texttt{S}_d)$ with $0 \leq \ell \leq d-1$, the \emph{opposite face} $f^* \in \Delta_{d-\ell-1}(\texttt{S}_d)$ is defined as the complement
\[
f^* := \texttt{S}_d \setminus f, \quad \text{so that} \quad f^*(\TT) = \TT \setminus f(\TT).
\]
This complement also admits a natural geometric realization
\[
f^*(T) = \operatorname{Convex}(\texttt{v}_{f^*(0)}, \ldots, \texttt{v}_{f^*(d - \ell - 1)}).
\]
Again we  use a single notation $f^*$ to refer to the abstract sub-simplex $f^* \in \Delta_{d-\ell-1}(\texttt{S}_d)$, its relabeled version $f^*(\TT) \in \Delta_{d-\ell-1}(\TT)$, and its geometric realization $f^*(T)$ if it is clear from the context.

In particular, let $F_i := \{i\}^* \in \Delta_{d-1}(\texttt{S}_d)$ denote the $(d-1)$-dimensional face opposite to the $i$-th vertex. Its geometric realization is given by the zero level set of $\lambda_i$:
\[
F_i(T) = \{x \in T : \lambda_i(x) = 0\}, \quad i = 0, \ldots, d.
\]
More generally, for $f \in \Delta_{\ell}(\texttt{S}_d)$, the geometric simplex $f(T)$ satisfies the identity
\[
f(T) = \bigcap_{i \in f^*} F_i (T) = \{x \in T : \lambda_i(x) = 0,\ i \in f^*\},
\]
which follows from the set-theoretic identity for the abstract simplicies:
\[
\bigcap_{i \in f^*} \{i\}^* = \left( \bigcup_{i \in f^*} \{i\} \right)^* = (f^*)^* = f.
\]

Let $f \in \Delta_\ell(\TT)$ and $T$ be a geometric realization of $\TT$. The sub-simplicial lattice $\mathbb{T}^d_k(f)$ denotes the subset of lattice points whose geometric realization lies in $f(T)$. To relate lattice indices across sub-simplices, we define the \emph{prolongation operator} $E_f : \mathbb{T}_s^\ell \to \mathbb{T}_s^d(f)$, which maps $\alpha \in \mathbb{T}_s^\ell$ to $\mathbb{T}_s^d(f)$ by
\[
E_f(\alpha)_{f(i)} = \alpha_i, \ i=0,1,\ldots, \ell, \qquad E_f(\alpha)_j = 0 \quad \text{for } j \in f^*.
\]
For example, if $f = \{1,2,5\} \subseteq \texttt{S}_5 = \{0,1,\ldots,5\}$ and $\alpha = (\alpha_0, \alpha_1, \alpha_2) \in \mathbb{T}_s^2$, then
\[
E_f(\alpha) = (0, \alpha_0, \alpha_1, 0, 0, \alpha_2) \in \mathbb{T}_s^5(f).
\]

Conversely, given $\alpha \in \mathbb{T}_k^d$ and $f \in \Delta_\ell(\texttt{S}_d)$, the \emph{restriction} $\alpha_f = R_f(\alpha) \in \mathbb{T}_s^\ell$, with $s = \sum_{i \in f} \alpha_i$, is defined component-wise as
\begin{equation}\label{eq:restriction}
(\alpha_f)_i =( R_f(\alpha))_i := \alpha_{f(i)}, \quad i = 0, \ldots, \ell.
\end{equation}
With a slight abuse of notation, $\alpha_f$ may also refer to its extension $E_f(\alpha_f)$. This leads to the decomposition
\[
\alpha = \alpha_f \oplus \alpha_{f^*} := E_f(\alpha_f) + E_{f^*}(\alpha_{f^*}), \quad |\alpha| = |\alpha_f| + |\alpha_{f^*}|.
\]
For example, for $\alpha = (\alpha_0, \alpha_1, \ldots, \alpha_5)$ and $f = \{1,2,5\}$,
\[
\begin{aligned}
\alpha_f &= (\alpha_1, \alpha_2, \alpha_5),  \quad E_f(\alpha_f) = (0, \alpha_1, \alpha_2, 0, 0, \alpha_5), \\
\alpha_{f^*} &= (\alpha_0, \alpha_3, \alpha_4), \quad E_{f^*}(\alpha_{f^*}) = (\alpha_0, 0, 0, \alpha_3, \alpha_4, 0).
\end{aligned}
\]
When the focus is on values, $\alpha_f$ and $E_f(\alpha_f)$ may be used interchangeably; when it is necessary to indicate the support explicitly, the notation $E_f(\alpha_f)$ will be used.

\subsection{Distance}
\begin{figure}[htbp]
\subfigure[Distance to an edge $f = \{1, 3\}$.]{
\begin{minipage}[t]{0.45\linewidth}
\centering
\includegraphics*[height=3.75cm]{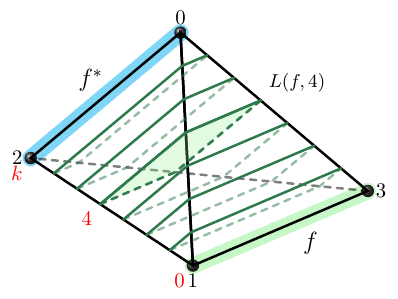}
\end{minipage}}
\qquad
\subfigure[Distance to a face $f = \{1, 2, 3\}$.]{
\begin{minipage}[t]{0.45\linewidth}
\centering
\includegraphics*[height=3.55cm]{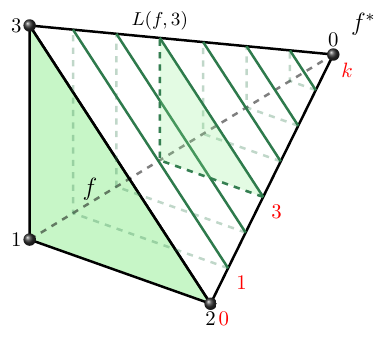}
\end{minipage}}
\caption{Illustration of $L(f, s) = \left\{ \alpha \in \mathbb{T}_k^d : \dist(\alpha, f) = s \right\}$ and $D(f, r) = \bigcup_{s = 0}^r L(f, s)$.}
\label{fig:dist}
\end{figure}

Given $f \in \Delta_\ell(\texttt{S}_d)$ with $0 \leq \ell \leq d-1$, the \emph{distance} from a lattice point $\alpha \in \mathbb{T}_k^d$ to $f$ is defined as
\[
\dist(\alpha, f) := |\alpha_{f^*}| = \sum_{i \in f^*} \alpha_i.
\]
We define the \emph{lattice tube} and the \emph{lattice layer} of $f$ by
\[
D(f, r) := \left\{ \alpha \in \mathbb{T}_k^d : \dist(\alpha, f) \le r \right\}, \qquad
L(f, s) := \left\{ \alpha \in \mathbb{T}_k^d : \dist(\alpha, f) = s \right\},
\]
so that
\[
D(f, r) = \bigcup_{s=0}^r L(f, s), \qquad L(f, s) = L(f^*, k - s).
\]
Under the geometric embedding to the reference simplex $\hat T$, the lattice points in $L(f, s)$ lie on the affine hyperplane
\[
x_{f^*(0)} + x_{f^*(1)} + \cdots + x_{f^*(d - \ell - 1)} = s.
\]
See Fig.~\ref{fig:dist} for illustrations of distance from lattice points to sub-simplices of different dimensions.

For a vertex $i \in \Delta_0(\texttt{S}_d)$, we have
\[
D(i, r) = \left\{ \alpha \in \mathbb{T}_k^d \mid  |\alpha_{i^*}| \le r \right\},
\]
which is combinatorially isomorphic to the degree-$r$ lattice $\mathbb{T}_r^d$; see the green triangles in Fig.~\ref{fig:lattice2D}. For $f \in \Delta_{d-1}(\texttt{S}_d)$, the geometric realization of $D(f, r)$ is a trapezoidal region of height $r$ with base face $f$. More generally, for any $f \in \Delta_\ell(\texttt{S}_d)$, the hyperplanes defined by $L(f, s)$ partition $\mathbb{T}_k^d(T)$ into two regions, and $D(f, r)$ corresponds to the region on the side containing $f$.

\begin{figure}[htbp]
\begin{center}
\includegraphics[width=4.025cm]{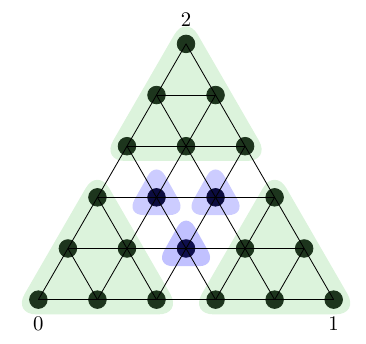}
\quad \quad
\includegraphics[width=4.25cm]{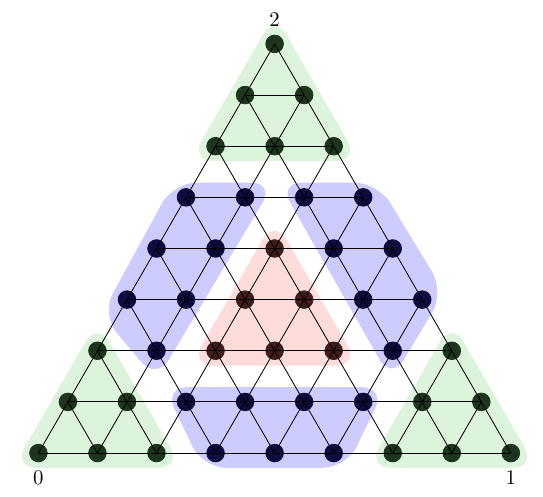}
\caption{For $\bs{r} = (2,1,0)$, two lattice decompositions of the simplicial lattice $\mathbb{T}_5^2$ (left) and $\mathbb{T}_8^2$ (right). The green triangles are $D(\textnormal{\texttt{v}}, 2)$. The purple trapezoid is $S_1(f)$, and the red triangle is $S_2(T)$.}
\label{fig:lattice2D}
\end{center}
\end{figure}

\subsection{Lattice Decomposition}
In~\cite{chen_geometric_2021} and~\cite[Appendix~A]{Chen;Huang:2022FEMcomplex3D}, a decomposition of the simplicial lattice $\mathbb{T}_k^d$ was introduced based on the distance structure.

\begin{theorem}\label{th:appdecT}
Let $m \geq 0$ be an integer, and let $\bs{r} = (r_0, r_1, \ldots, r_d)$ be a sequence of integers satisfying
\[
r_d = 0, \quad r_{d-1} = m, \quad r_\ell \geq 2 r_{\ell+1} \quad \text{for } \ell = d-2, \ldots, 0.
\]
Assume that the degree $k$ satisfies $k \geq 2r_0 + 1 \geq 2^d m + 1$. Then the lattice $\mathbb{T}_k^d$ admits the direct sum decomposition
\begin{equation}\label{eq:appdecT}
\begin{aligned}
\mathbb{T}_k^d &= \Oplus_{\ell=0}^d \Oplus_{f \in \Delta_\ell(\TT)} S_\ell(f), \\
\text{where} \qquad S_\ell(f) &= D(f, r_\ell) \setminus \left( \bigcup_{i=0}^{\ell-1} \bigcup_{e \in \Delta_i(f)} D(e, r_i) \right).
\end{aligned}
\end{equation}
\end{theorem}

Here, $S_\ell(f)$ consists of the lattice points that are within distance $r_\ell$ to $f$ but outside the distance tubes of all lower-dimensional sub-simplices contained in $f$. The decomposition~\eqref{eq:appdecT} partitions the entire lattice into disjoint sets, each associated with a face $f$ of dimension $\ell$. We refer to Fig.~\ref{fig:lattice2D} for  decompositions of the simplicial lattice $\mathbb{T}_5^2$ and $\mathbb{T}_8^2$ with $\bs{r} = (2,1,0)$.

\subsection{Triangulation and Simplicial Complex}

Let $\Omega \subset \mathbb{R}^d$ be a polyhedral domain with $d \geq 1$. A geometric triangulation $\mathcal{T}_h$ of $\Omega$ is a collection of $d$-simplices such that
\[
\bigcup_{T \in \mathcal{T}_h} T = \Omega, \qquad \mathring{T}_i \cap \mathring{T}_j = \emptyset \quad \text{for all } T_i \neq T_j \in \mathcal{T}_h,
\]
where $\mathring{T}$ denotes the interior of the simplex $T$. The subscript $h$ represents the mesh size, i.e., the maximum diameter of all elements. In this work, we restrict our attention to \emph{conforming triangulations}, where the intersection of any two simplices is either empty or a common sub-simplex of lower dimension.

While classical finite element methods operate primarily on geometric triangulations, we adopt a topological perspective using simplicial complexes from algebraic topology~\cite{hatcher2005algebraic} to better formalize the underlying combinatorial structure.

A \emph{simplicial complex} $\mathcal{S}$ over a finite vertex set $V$ is a collection of subsets of $V$ such that if $\TT \in \mathcal{S}$ is a $d$-simplex, then all its sub-simplices $\Delta(\TT)$ are also in $\mathcal{S}$. Elements of $V$ are referred to as \emph{vertices}, and elements of $\mathcal{S}$ are called \emph{simplices}. We denote by $\Delta_{\ell}(\mathcal{S})$ the set of all $\ell$-simplices in $\mathcal{S}$. A simplex $\TT \in \mathcal{S}$ is called \emph{maximal} if it is not a proper subset of any other simplex in $\mathcal{S}$. The complex $\mathcal{S}$ is said to be \emph{pure of dimension $d$} if all maximal simplices are $d$-simplices.

Without loss of generality, we let $V = \{1, 2, \ldots, N\}$ with $N \geq d+1$. Here, vertices are considered as abstract entities. A $d$-dimensional abstract simplicial complex $\mathcal{S}$ can be represented by a matrix \texttt{elem(1:NT, 0:d)}, where \texttt{NT} is the number of elements. Each row \texttt{elem(t, 0:d)} is an abstract $d$-simplex, consists of the global vertex indices of the $d$-simplex $t$, and $\texttt{S}_d = \{0, 1, \ldots, d\}$ serves as the local index set. Since different permutations of vertices represent the same abstract simplex, ordering becomes relevant when managing global degrees of freedom.

The matrix \texttt{node(1:N, 1:d)} provides a geometric realization of the abstract vertices, which will induce a geometric realization of the simplicial complex. For instance, in 2D, \texttt{node(k, 1:2)} stores the $x$- and $y$-coordinates of the $k$-th vertex. We refer the reader to~\cite{Chen2018Programming} for a concise introduction to the \texttt{node} and \texttt{elem} data structures, and to~\cite{Chen2008ifem} for detailed discussions on indexing, ordering, and orientation via the \texttt{sc} and \texttt{sc3} documentation in $i$FEM.

The geometric realization of the maximal simplices $\Delta_d(\mathcal{S})$ yields the geometric triangulation $\mathcal{T}_h$. Throughout this paper, we follow the finite element convention of working directly with $\mathcal{T}_h$, and adopt the notation $\Delta_\ell(\mathcal{T}_h)$ to denote the set of all $\ell$-simplices in the triangulation.

The lattice points over a conforming triangulation $\mathcal{T}_h$ are given by the union
\[
\mathbb{T}_k^d(\mathcal{T}_h) := \bigcup_{T \in \mathcal{T}_h} \mathbb{T}_k^d(T).
\]
Note that this union contains duplicate entries: for instance, a sub-simplex $f \in \Delta_\ell(\mathcal{T}_h)$ may belong to multiple elements $T \in \mathcal{T}_h$, and the corresponding sets $\mathbb{T}_k^d(f)$ are counted repeatedly. A disjoint version of the lattice is given by
\begin{equation}\label{eq:Xsum}
\mathbb{T}_k^d(\mathcal{T}_h) = \Oplus_{\ell = 0}^d \Oplus_{f \in \Delta_\ell(\mathcal{T}_h)} \mathbb{T}_k^\ell(\mathring{f}),
\end{equation}
where $\mathbb{T}_k^\ell(\mathring{f})$ denotes the set of lattice points whose geometric embeddings lie in the interior of the realization of $f$ and for $\ell = 0$, $\mathring{f} = f$. 

For practical implementations, a unique global indexing of lattice points associated with vertices, edges, faces, and interiors is required. This involves constructing a mapping from local indices to global indices, which is further discussed in Section~\ref{sec:global}.

\section{Bernstein Basis and Derivatives}\label{sec:B}

In this section, we introduce the Bernstein basis for the polynomial space $\mathbb{P}_k(T)$ over a simplex $T$ and derive explicit formulas for its integrals and derivatives.

\subsection{Bernstein Basis}

The \emph{Bernstein basis} for $\mathbb{P}_k(T)$ is defined by
\[
\mathcal{B}_k := \left\{ B^{\alpha} := \frac{k!}{\alpha!} \lambda^{\alpha} : \alpha \in \mathbb{T}_k^d \right\},
\]
where $\lambda = (\lambda_0, \ldots, \lambda_d)$ are the barycentric coordinates associated with the vertices of $T$, and $\lambda^{\alpha} = \lambda_0^{\alpha_0} \cdots \lambda_d^{\alpha_d}$.

An important property of the Bernstein polynomials is the explicit formula for their integral over $T$ (cf.~\cite{ChenFEM}):
\begin{equation}\label{eq:int-bernstein}
\int_T B^{\alpha} \dx = \frac{k! \, d!}{(k+d)!} |T|,
\end{equation}
where $|T|$ denotes the $d$-dimensional volume of $T$.

For a subsimplex $f \in \Delta_\ell(\TT)$, we define the \emph{Bernstein basis} over $f$ as
\[
\mathcal{B}_k(f) := \left\{ B^{\alpha}_f := \frac{k!}{\alpha!} \lambda_f^{\alpha} : \alpha \in \mathbb{T}_k^\ell \right\},
\]
where $\lambda_f = (\lambda_{f(0)}, \ldots, \lambda_{f(\ell)})$ are the barycentric coordinates associated with the vertices of $f$, and $\lambda_f^{\alpha} = \prod_{i=0}^{\ell} \lambda_{f(i)}^{\alpha_i}$.

The correspondence between lattice points $\alpha \in \mathbb{T}_k^d$ and Bernstein polynomials $B^\alpha$ allows many properties of polynomials to be understood directly through the structure of the simplicial lattice.

\subsection{Tensors}

We use the standard Euclidean inner product to identify $(\mathbb{R}^d)' \cong \mathbb{R}^d$, and present all tensor operations without explicit reference to the dual space.

For integers $r, d \geq 1$, the $r$-th order tensor space over $\mathbb{R}^d$ is defined as
\[
\mathbb{R}^{d,r} := (\mathbb{R}^d)^{\otimes r} = \underbrace{\mathbb{R}^d \otimes \cdots \otimes \mathbb{R}^d}_{r \text{ times}}.
\]
The standard inner product on $\mathbb{R}^d$ extends naturally to $\mathbb{R}^{d,r}$, and is denoted by the symbol $:$. For elementary tensors, we have
\begin{equation}\label{eq:innerproduct1}
(\boldsymbol{t}_1 \otimes \cdots \otimes \boldsymbol{t}_r) :
(\boldsymbol{n}_1 \otimes \cdots \otimes \boldsymbol{n}_r) =
\prod_{i=1}^r \boldsymbol{t}_i \cdot \boldsymbol{n}_i,
\end{equation}
where $\cdot$ denotes the standard inner product in $\mathbb{R}^d$.

Let $\{\boldsymbol{t}_1, \ldots, \boldsymbol{t}_d\}$ be a basis of $\mathbb{R}^d$, and let $\{\hat{\boldsymbol{t}}_1, \ldots, \hat{\boldsymbol{t}}_d\} \subset \mathbb{R}^d$ denote its dual basis, satisfying
\[
\hat{\boldsymbol{t}}_i \cdot \boldsymbol{t}_j = \delta_{i,j}, \qquad \text{for } 1 \le i, j \le d,
\]
where $\delta_{i,j}$ is the Kronecker delta. Then any tensor $\boldsymbol{\tau} \in \mathbb{R}^{d,r}$ can be written as
\[
\boldsymbol{\tau} = \tau_{i_1 \cdots i_r} \, \boldsymbol{t}_{i_1} \otimes \cdots \otimes \boldsymbol{t}_{i_r},
\]
where repeated indices are summed using Einstein notation. The coefficients are given by the inner product with the dual basis
\[
\tau_{i_1 \cdots i_r} = \boldsymbol{\tau} : \left( \hat{\boldsymbol{t}}_{i_1} \otimes \cdots \otimes \hat{\boldsymbol{t}}_{i_r} \right).
\]
When the canonical orthonormal basis $\{\boldsymbol{e}_i\}_{i=1}^d$ is used, we simply write $\boldsymbol{\tau} = (\tau_{i_1 \cdots i_r})$ as a $d^r$-dimensional array.

The change of coordinates follows directly from the definition.
\begin{lemma}
\label{lem:tauexpressexchbasis}
Let $\{\boldsymbol{s}_1, \ldots, \boldsymbol{s}_d\}$ be a basis of $\mathbb{R}^d$, and let $\{\hat{\boldsymbol{s}}_1, \ldots, \hat{\boldsymbol{s}}_d\} \subset \mathbb{R}^d$ denote its dual basis.
Consider a tensor $\boldsymbol{\tau} \in \mathbb{R}^{d,r}$ represented as
\[
\boldsymbol{\tau} = \tau_{i_1 \cdots i_r} \, \boldsymbol{t}_{i_1} \otimes \cdots \otimes \boldsymbol{t}_{i_r} = \widetilde{\tau}_{i_1 \cdots i_r} \, \boldsymbol{s}_{i_1} \otimes \cdots \otimes \boldsymbol{s}_{i_r}.
\]
Then its components satisfy
\begin{equation}\label{eq:tauexpressexchbasis}
\widetilde{\tau}_{i_1 \cdots i_r} = \boldsymbol{\tau} : \left( \hat{\boldsymbol{s}}_{i_1} \otimes \cdots \otimes \hat{\boldsymbol{s}}_{i_r} \right)= \tau_{j_1 \cdots j_r} \, (\boldsymbol{t}_{j_1} \cdot \hat{\boldsymbol{s}}_{i_1}) \cdots (\boldsymbol{t}_{j_r} \cdot \hat{\boldsymbol{s}}_{i_r}).
\end{equation}
\end{lemma}

To index tensor monomials, define the increasing multi-index set
\[
\mathcal{I}_\ell^r := \left\{ (i_1, \ldots, i_r) \in \{1, \ldots, \ell\}^r : i_1 \le i_2 \le \cdots \le i_r \right\}.
\]
There exists a natural one-to-one correspondence between $\mathcal{I}_\ell^r$ and the simplicial lattice $\mathbb{T}_r^{\ell - 1}$: for $\alpha = (\alpha_1, \ldots, \alpha_\ell) \in \mathbb{T}_r^{\ell - 1}$, define
\begin{equation}\label{eq:alpha2i}
I^{\alpha} := (\underbrace{1, \ldots, 1}_{\alpha_1}, \ldots, \underbrace{\ell, \ldots, \ell}_{\alpha_\ell}) \in \mathcal{I}_\ell^r.
\end{equation}
Namely $\alpha_i$ is the number of index $i$ appearing in $I^{\alpha}$.

Given $\alpha \in \mathbb{T}_r^{\ell-1}$ and a vector array $t = (\boldsymbol{t}_1, \ldots, \boldsymbol{t}_\ell)$ with $\boldsymbol{t}_i \in \mathbb{R}^d$, define the tensor monomial
\[
t^{\alpha} := t^{\otimes \alpha} := \boldsymbol{t}_1^{\otimes \alpha_1} \otimes \cdots \otimes \boldsymbol{t}_\ell^{\otimes \alpha_\ell}
= \boldsymbol{t}_{I^\alpha_1} \otimes \boldsymbol{t}_{I^\alpha_2} \otimes \cdots \otimes \boldsymbol{t}_{I^\alpha_r} \in \mathbb{R}^{d,r},
\]
where
\[
\boldsymbol{t}_i^{\otimes \alpha_i} := \underbrace{\boldsymbol{t}_i \otimes \cdots \otimes \boldsymbol{t}_i}_{\alpha_i \text{ times}}.
\]
Note that the vectors $\{\boldsymbol{t}_1, \ldots, \boldsymbol{t}_\ell\}$ are not required to be linearly independent, and the length $\ell$ may exceed the ambient dimension $d$. Moreover, the tensor product $t^{\alpha}$ depends on the ordering of the vectors; for example, $\boldsymbol{t}_1 \otimes \boldsymbol{t}_2 \neq \boldsymbol{t}_2 \otimes \boldsymbol{t}_1$ in general.

\subsection{Symmetric Tensors}
We introduce the $r$-th order symmetric tensor space over $\mathbb{R}^d$ as follows:
\[
\mathbb{S}^{d,r} := \left\{ \boldsymbol{\tau} = \tau_{i_1 \cdots i_r} \boldsymbol{t}_{i_1} \otimes \cdots \otimes \boldsymbol{t}_{i_r} \in \mathbb{R}^{d,r} : \tau_{i_{\sigma(1)} \cdots i_{\sigma(r)}} = \tau_{i_1 \cdots i_r} \text{ for any } \sigma \in \mathcal{G}^r \right\},
\]
where $\mathcal{G}^r$ denotes the permutation group of $(1, \ldots, r)$. We show that the symmetry is intrinsic and independent of the choice of basis.

\begin{lemma}
Under the same assumptions as in Lemma~\ref{lem:tauexpressexchbasis},  
let $\boldsymbol{\tau} \in \mathbb{S}^{d,r}$ be a symmetric tensor expressed as
\[
\boldsymbol{\tau} = \widetilde{\tau}_{i_1 \cdots i_r} \, \boldsymbol{s}_{i_1} \otimes \cdots \otimes \boldsymbol{s}_{i_r}.
\]
Then the components $\widetilde{\tau}_{i_1 \cdots i_r}$ remain symmetric under any permutation $\sigma \in \mathcal{G}^r$, i.e.,
\[
\widetilde{\tau}_{i_{\sigma(1)} \cdots i_{\sigma(r)}} = \widetilde{\tau}_{i_1 \cdots i_r}, \qquad \forall~\sigma \in \mathcal{G}^r, \; 1 \leq i_1, \ldots, i_r \leq d.
\]
\end{lemma}
\begin{proof}
By \eqref{eq:tauexpressexchbasis} and the symmetry of $\tau_{j_1 \cdots j_r}$,
\begin{align*}
\widetilde{\tau}_{i_{\sigma(1)} \cdots i_{\sigma(r)}} & = \boldsymbol{\tau} : \left( \hat{\boldsymbol{s}}_{i_{\sigma(1)}} \otimes \cdots \otimes \hat{\boldsymbol{s}}_{i_{\sigma(r)}} \right)= \tau_{j_{\sigma(1)} \cdots j_{\sigma(r)}} \, (\boldsymbol{t}_{j_{\sigma(1)}} \cdot \hat{\boldsymbol{s}}_{i_{\sigma(1)}}) \cdots (\boldsymbol{t}_{j_{\sigma(r)}} \cdot \hat{\boldsymbol{s}}_{i_{\sigma(r)}}) \\
&= \tau_{j_1 \cdots j_r} \, (\boldsymbol{t}_{j_1} \cdot \hat{\boldsymbol{s}}_{i_1}) \cdots (\boldsymbol{t}_{j_r} \cdot \hat{\boldsymbol{s}}_{i_r}) = \widetilde{\tau}_{i_1 \cdots i_r}.
\end{align*}
\end{proof}

The symmetrization operator ${\rm sym}(\boldsymbol{\tau})$ for a tensor $\boldsymbol{\tau} \in \mathbb{R}^{d,r}$ is defined as
\[
{\rm sym}(\boldsymbol{\tau})_{i_1 \cdots i_r} = \frac{1}{r!} \sum_{\sigma \in \mathcal{G}^r} \tau_{i_{\sigma(1)} \cdots i_{\sigma(r)}}, \quad 1 \leq i_1, \ldots, i_r \leq d.
\]

In particular, when $\boldsymbol{\tau} = t^{\otimes \alpha}$ for some $\alpha \in \mathbb{T}_r^{\ell-1}$, the symmetrization admits a simplified form. Let $I^\alpha \in \mathcal{I}_\ell^r$ be the increasing multi-index associated with $\alpha$, as defined in~\eqref{eq:alpha2i}. 
For each permutation $\sigma \in \mathcal{G}^r$, define $(\sigma(I^\alpha))_i := I^\alpha_{\sigma(i)}$. This induces an equivalence relation $\sim^\alpha$ on $\mathcal{G}^r$:
\[
\sigma \sim^\alpha \sigma' \iff \sigma(I^\alpha) = \sigma'(I^\alpha).
\]
Let $\mathcal{G}^r_\alpha := \mathcal{G}^r / \sim^\alpha$ denote the set of equivalence classes under this relation. Each equivalence class has cardinality $\alpha!$, since permuting the $\alpha_i$ identical items in $\boldsymbol{t}_i^{\otimes \alpha_i}$ results in equivalent terms. This leads to the simplified expression
\begin{equation}
\label{eq:sym}
\operatorname{sym}(t^{\otimes \alpha}) = \frac{\alpha!}{r!} \sum_{\sigma \in \mathcal{G}^r_\alpha} \boldsymbol{t}_{I^\alpha_{\sigma(1)}} \otimes \cdots \otimes \boldsymbol{t}_{I^\alpha_{\sigma(r)}}.
\end{equation}

\begin{lemma}\label{lm:Sbasis}
Let $\{\boldsymbol{t}_1, \dots, \boldsymbol{t}_d\}$ be a basis of $\mathbb{R}^d$, and let $\{\hat{\boldsymbol{t}}_1, \dots, \hat{\boldsymbol{t}}_d\} \subset \mathbb{R}^d$ be its dual basis. Then the set $\{\operatorname{sym}({t}^{\otimes \alpha})\}_{\alpha \in \mathbb{T}_r^{d-1}}$ and $\{\operatorname{sym}(\hat{{t}}^{\otimes \alpha})\}_{\alpha \in \mathbb{T}_r^{d-1}}$ are scaled dual bases for the symmetric tensor space $\mathbb{S}^{d, r}$. In particular, the following duality relation holds:
\begin{equation}\label{eq:tensordual}
\operatorname{sym}(\hat{{t}}^{\otimes \alpha}) :
\operatorname{sym}({t}^{\otimes \beta}) = \frac{\alpha!}{r!} \delta_{\alpha, \beta}, \quad \alpha, \beta \in \mathbb{T}_r^{d-1}.
\end{equation}
\end{lemma}

\begin{proof}
It is clear that the set $\{\boldsymbol{t}_{i_1} \otimes \cdots \otimes \boldsymbol{t}_{i_r}\}_{1 \leq i_1, \dots, i_r \leq d}$ forms a basis for $\mathbb{R}^{d,r}$. Each basis element corresponds to an index array $(i_1, \ldots, i_r)$, and by sorting the indices in increasing order, we obtain a multi-index $\mathrm{sort}(i_1, \ldots, i_r) \in \mathcal{I}_r^{d-1}$, which corresponds to a unique $\alpha \in \mathbb{T}_r^{d-1}$ such that $\mathrm{sort}(i_1, \ldots, i_r) = I^\alpha$. Hence, we have
\[
\operatorname{sym}(\boldsymbol{t}_{i_1} \otimes \cdots \otimes \boldsymbol{t}_{i_r}) = \operatorname{sym}({t}^{\otimes \alpha}), \quad \text{whenever } \mathrm{sort}(i_1, \ldots, i_r) = I^\alpha.
\]
This shows that $\{\operatorname{sym}({t}^{\otimes \alpha})\}_{\alpha \in \mathbb{T}_r^{d-1}}$ spans the symmetric tensor space $\mathbb{S}^{d, r}$.

To prove the duality, we compute the inner product using the symmetrization definition \eqref{eq:sym} and the standard inner product \eqref{eq:innerproduct1}:
\[
\operatorname{sym}(\hat{{t}}^{\otimes \alpha}) :
\operatorname{sym}({t}^{\otimes \beta}) = \frac{\alpha!}{r!} \cdot \frac{\beta!}{r!} \sum_{\sigma \in \mathcal{G}_\alpha^r} \sum_{\sigma' \in \mathcal{G}_\beta^r} \prod_{i=1}^r \hat{\boldsymbol{t}}_{I^\alpha_{\sigma(i)}} \cdot \boldsymbol{t}_{I^\beta_{\sigma'(i)}}.
\]
Note that
\[
\hat{\boldsymbol{t}}_{I^\alpha_{\sigma(i)}} \cdot \boldsymbol{t}_{I^\beta_{\sigma'(i)}} = \delta_{I^\alpha_{\sigma(i)}, I^\beta_{\sigma'(i)}},
\]
so the product
\[
\prod_{i=1}^r \hat{\boldsymbol{t}}_{I^\alpha_{\sigma(i)}} \cdot \boldsymbol{t}_{I^\beta_{\sigma'(i)}}
\]
is equal to $1$ if and only if $\sigma(I^\alpha) = \sigma'(I^\beta)$, and $0$ otherwise.

If $\alpha \neq \beta$, there is no pair $(\sigma, \sigma')$ such that $\sigma(I^\alpha) = \sigma'(I^\beta)$, so the entire sum vanishes. If $\alpha = \beta$, then $I^\alpha = I^\beta$, and for each $\sigma \in \mathcal{G}_\alpha^r$, there exists a unique $\sigma' = \sigma \in \mathcal{G}_\alpha^r$ such that $\sigma(I^\alpha) = \sigma'(I^\beta)$. The sum becomes
\[
\sum_{\sigma \in \mathcal{G}_\alpha^r} \sum_{\sigma' \in \mathcal{G}_\alpha^r} \prod_{i=1}^r \delta_{I^\alpha_{\sigma(i)}, I^\alpha_{\sigma'(i)}} = \sum_{\sigma \in \mathcal{G}_\alpha^r} 1 = |\mathcal{G}_\alpha^r| = \frac{r!}{\alpha!}.
\]
Thus, we obtain
\[
\operatorname{sym}(\hat{{t}}^{\otimes \alpha}) :
\operatorname{sym}({t}^{\otimes \beta}) = \frac{\alpha!}{r!} \cdot \frac{\alpha!}{r!} \cdot \frac{r!}{\alpha!} = \frac{\alpha!}{r!}.
\]

This duality relation implies the linear independence of the set $\{\operatorname{sym}({t}^{\otimes \alpha})\}_{\alpha \in \mathbb{T}_r^{d-1}}$, and thus it forms a basis for the symmetric tensor space $\mathbb{S}^{d, r}$. The dual basis is given by $\left\{ \frac{r!}{\alpha!} \operatorname{sym}(\hat{{t}}^{\otimes \alpha}) \right\}_{\alpha \in \mathbb{T}_r^{d-1}}$.
\end{proof}

As a byproduct, we obtain the following identity for the dimension of the symmetric tensor space:
\[
\dim \mathbb{S}^{d, r} = |\mathbb{T}_r^{d-1}| = {d+r-1 \choose r} \ll d^r = \dim \mathbb{R}^{d, r}.
\]

\begin{corollary}\label{cor:symsym}
We have the identity
\begin{equation}\label{eq:symsym}
\boldsymbol{\tau} : \operatorname{sym}(\boldsymbol{\varsigma}) = \boldsymbol{\tau} : \boldsymbol{\varsigma}, \quad \forall~\boldsymbol{\tau} \in \mathbb{S}^{d,r}, \ \boldsymbol{\varsigma} \in \mathbb{R}^{d,r}.
\end{equation}
\end{corollary}

\begin{proof}
Let $\boldsymbol{\tau} \in \mathbb{S}^{d,r}$ be a symmetric tensor. For $1 \leq i_1, \dots, i_r \leq d$, we compute:
\begin{align*}
\boldsymbol{\tau} : \operatorname{sym}(\hat{\boldsymbol{t}}_{i_1} \otimes \cdots \otimes \hat{\boldsymbol{t}}_{i_r}) &= \frac{1}{r!} \sum_{\sigma \in \mathcal{G}^r} \boldsymbol{\tau} : (\hat{\boldsymbol{t}}_{i_{\sigma(1)}} \otimes \cdots \otimes \hat{\boldsymbol{t}}_{i_{\sigma(r)}}) \\
&= \frac{1}{r!} \sum_{\sigma \in \mathcal{G}^r} \tau_{i_{\sigma(1)} \cdots i_{\sigma(r)}} = \tau_{i_1 \cdots i_r} = \boldsymbol{\tau} : (\hat{\boldsymbol{t}}_{i_1} \otimes \cdots \otimes \hat{\boldsymbol{t}}_{i_r}).
\end{align*}
Since $\{\hat{\boldsymbol{t}}_{i_1} \otimes \cdots \otimes \hat{\boldsymbol{t}}_{i_r}: 1 \leq i_1, \dots, i_r \leq d\}$ forms a basis for the tensor space $\mathbb{R}^{d,r}$, we conclude the identity in~\eqref{eq:symsym}.
\end{proof}

We refer to \cite{Wasserman2009} for more discussion on tensors and symmetric tensors.

\subsection{Derivatives}

Due to the commutativity of differentiation, the $r$-th order derivative of a function $v$ is a symmetric tensor. Let $\{\boldsymbol{e}_i\}_{i=1}^d$ be the canonical orthonormal basis of $\mathbb{R}^d$, corresponding to the coordinate system $x = (x_1, x_2, \ldots, x_d)$. For a function $v \in H^r(K)$, the $r$-th order derivative is defined as:
\[
    \nabla^r v = \frac{\partial^r v}{\partial x_{i_1} \cdots \partial x_{i_r}} \boldsymbol{e}_{i_1} \otimes \cdots \otimes \boldsymbol{e}_{i_r} \in \mathbb{S}^{d, r}, \quad 1 \leq i_1, \ldots, i_r \leq d.
\]
This can also be expressed using multi-index notation as:
\[
    \partial^{\alpha} v = \frac{\partial^r v}{\partial x_1^{\alpha_1} \cdots \partial x_d^{\alpha_d}} = \nabla^r v : e^{\otimes \alpha}, \quad \alpha \in \mathbb{T}_r^{d-1}.
\]

In general, for a set of linearly independent vectors $n = \{\boldsymbol{n}_1, \dots, \boldsymbol{n}_\ell\}$ in $\mathbb{R}^d$ and $\alpha \in \mathbb{T}^{\ell-1}_r$, we define:
\[
    \frac{\partial^r v}{\partial n^\alpha}  := \nabla^r v : n^{\otimes \alpha}.
\]

\begin{lemma}
    Let $n = \{\boldsymbol{n}_1, \dots, \boldsymbol{n}_d\}$ be a basis of $\mathbb{R}^d$, and $\hat n = \{\hat{\boldsymbol{n}}_1, \dots, \hat{\boldsymbol{n}}_d\}$ be its dual basis. For a smooth function $v$ and multi-index $\alpha$ of order $r$, the $r$-th order derivative of $v$ can be expressed as:
    \begin{equation}
        \nabla^r v = \sum_{\alpha \in \mathbb{T}_r^{d-1}} \frac{r!}{\alpha!} \operatorname{sym}(\hat{{n}}^{\otimes \alpha}) \frac{\partial^{r} v}{\partial n^{\alpha}},
        \label{eq:nablarv}
    \end{equation}
where $\frac{\partial^r v}{\partial n^\alpha} := \nabla^r v : n^{\otimes \alpha}$.
\end{lemma}

\begin{proof}
By Lemma \ref{lm:Sbasis}, the set $\{\operatorname{sym}(\hat{{n}}^{\otimes \alpha})\}_{\alpha \in \mathbb{T}_r^{d-1}}$ forms a basis for the symmetric tensor space $\mathbb{S}^{d, r}$. Thus, we can express $\nabla^r v$ as a linear combination:
\[
    \nabla^r v = \sum_{\alpha \in \mathbb{T}_r^{d-1}} c_{\alpha} \operatorname{sym}(\hat{{n}}^{\otimes \alpha}),
\]
where the coefficients $c_{\alpha}$ can be computed by the duality relation~\eqref{eq:tensordual}. 

Using the fact that $\hat{\boldsymbol{n}}_i \cdot \boldsymbol{n}_j = \delta_{i,j}$, we have:
\[
    \hat{{n}}^{\otimes \alpha} : {n}^{\otimes \beta} = \delta_{\alpha, \beta},
    \quad \operatorname{sym}(\hat{{n}}^{\otimes \alpha}) : {n}^{\otimes \beta} = \operatorname{sym}(\hat{{n}}^{\otimes \alpha}) :  \operatorname{sym}({n}^{\otimes \beta}) = \frac{\alpha!}{r!} \delta_{\alpha, \beta}.
\]
Thus, we get:
\[
    \frac{\partial^{r} v}{\partial n^{\alpha}} = \nabla^r v : n^{\otimes \alpha} = c_{\alpha} \operatorname{sym}(\hat{{n}}^{\otimes \alpha}) : n^{\otimes \alpha} = c_{\alpha} \frac{\alpha!}{r!}.
\]
This gives the formula for the coefficients $c_{\alpha}$, as stated in~\eqref{eq:nablarv}.
\end{proof}

We present the derivatives of Bernstein polynomials in the following lemma. Similar formulae for higher-order directional derivatives of Bernstein polynomials can be found in \cite[(17.17)]{Farin2002}, \cite[Theorem 2.13]{LaiSchumaker2007}, and \cite[(3.16)]{ainsworth2011bernstein}. However, the formulation presented below provides a more complete and comprehensive version, extending the previous results to include the full range of derivative orders.

\begin{lemma}
Let $B^{\beta}$ be a Bernstein polynomial for $\beta \in \mathbb{T}_k^d$. For $0 \leq r \leq k$, the $r$-th order derivative of $B^{\beta}$ is given by:
\begin{equation}
\nabla^r B^{\beta} = \sum_{\alpha \in \mathbb{T}_r^d, \alpha \leq \beta} \frac{r! k!}{(k-r)! \alpha!} \operatorname{sym}((\nabla\lambda)^{\otimes \alpha}) B^{\beta - \alpha},
\label{eq:bern_grad}
\end{equation}
where $\nabla \lambda = (\nabla \lambda_0, \nabla \lambda_1, \dots, \nabla \lambda_d)$.
\end{lemma}

\begin{proof}
Define $\epsilon_i \in \mathbb{T}_1^d$ as $\epsilon_i = (0, \ldots, 1, \ldots, 0)$, where the 1 appears in the $i$-th position. Then, the first-order gradient of $B^{\beta}$ can be written as:
\[
\nabla B^{\beta} = \sum_{\epsilon_i \in \mathbb{T}_1^d, \epsilon_i \leq \beta} \frac{k!}{(\beta - \epsilon_i)!} \lambda^{\beta - \epsilon_i} (\nabla \lambda_i),
\]
which leads to:
\[
\nabla^r B^{\beta} = \sum_{\epsilon_{i_1}, \dots, \epsilon_{i_r} \in \mathbb{T}_1^d, \alpha = \epsilon_{i_1} + \cdots + \epsilon_{i_r} \leq \beta} \frac{k!}{(\beta - \alpha)!} \lambda^{\beta - \alpha} (\nabla \lambda)^{\otimes \alpha}.
\]
Due to the symmetry of $\nabla^r B^{\beta}$, we can rewrite the above equation as:
\begin{align*}
\nabla^r B^{\beta} &= \sum_{\alpha \in \mathbb{N}^{0:d}, |\alpha| = r, \alpha \leq \beta} \frac{k!}{(\beta - \alpha)!} \lambda^{\beta - \alpha} \operatorname{sym}((\nabla \lambda)^{\otimes \alpha}) \\
&= \sum_{\alpha \in \mathbb{T}_r^d, \alpha \leq \beta} \frac{r! k!}{(k - r)! \alpha!} \operatorname{sym}((\nabla \lambda)^{\otimes \alpha}) B^{\beta - \alpha}.
\end{align*}
\end{proof}

\subsection{Derivative and Distance}
Recall that in \cite{ArnoldFalkWinther2009}, a smooth function $u$ is said to vanish to order $r$ on a sub-simplex $f$ if $\nabla^{\alpha} u|_f = 0$ for all $\alpha \in \mathbb{N}^{1:d}$, with $|\alpha| < r$. The following result establishes a relationship between the vanishing order of a Bernstein polynomial $B^{\beta}$ on $f$ and the distance $\dist(\beta, f)$.

\begin{lemma}\label{lm:derivative}
Let $T$ be a $d$-dimensional simplex, $f \in \Delta_{\ell}(\TT)$, and $\beta \in \mathbb{T}_k^d(T)$. Then
\[
(\nabla^r B^{\beta})|_f = 0 \quad \text{for} \quad 0 \leq r < \dist(\beta, f).
\]
\end{lemma}

\begin{proof}
By assumption, $\dist(\beta, f) = |\beta_{f^*}| > r$. For each $\alpha \in \mathbb{T}_r^d$ such that $\alpha \leq \beta$, we have
\[
|(\beta - \alpha)_{f^*}| = |\beta_{f^*}| - |\alpha_{f^*}| \geq |\beta_{f^*}| - r > 0.
\]
Thus, $B^{\beta - \alpha}$ contains a factor of $\lambda_i$ for some $i \in f^*$, and consequently, $B^{\beta - \alpha}|_f = 0$ by the property that $\lambda_i|_f = 0$ for $i \in f^*$. Using the expression for $\nabla^r B^{\beta}$ from \eqref{eq:bern_grad}, the desired result follows.
\end{proof}

Consider the 1-dimensional reference simplex $T = [0,1]$ and the sub-simplex $f = \{ 0 \}$, the left vertex. Then, $\lambda_0 = 1 - x$ and $\lambda_1 = x$ and 
\[
\nabla^{r} \left( (1 - x)^{\alpha_f} x^{\alpha_{f^*}} \right) \big|_{x = 0} = 0 \quad \text{if} \quad |\alpha_{f^*}| > r.
\]
Lemma \ref{lm:derivative} is a generalization of this 1-D result to a simplex in multi-dimensions. 

\section{Local Frame, Degrees of Freedom, and Basis}
\label{sec:local_frame}
In this section, we present element-wise degrees of freedom (DoFs) for smooth finite elements, and find out its dual basis using Bernstein basis.

\subsection{Construction of Dual Bases}
We can give a linear indexing of $\mathbb T_k^d$ and consequently $\mathcal B_k$, e.g., the 
dictionary ordering:
\[
\alpha \to \sum_{i=1}^d{\alpha_i+\alpha_{i+1}+\cdots+\alpha_d+d-i \choose d+1-i}.
\]
With such indexing, we can treat $\mathcal B_k = (B^{\alpha})$ as a vector of basis polynomials. 
Denote by $(D_{\alpha, \beta})$ or $(D^{\alpha, \beta})$ the matrix using the linear indexing of the first subscript as the row index and the second as the column index. Then the transpose of $(D_{\alpha, \beta})$ is $(D_{\beta, \alpha})$.

\begin{lemma}\label{lm:constructdualbases}
Let $\mathcal{L} = \{\textnormal{\texttt{l}}^{\alpha}\mid \alpha \in \mathbb{T}_k^d\}$ be a basis of $\mathbb{P}_k(T)'$. Then, $(\textnormal{\texttt{l}}^{\alpha}(B^{\beta}))^{-1}\mathcal{L}$ is a basis of $\mathbb{P}_k(T)'$ dual to $\mathcal{B}_k$, and 
$(\textnormal{\texttt{l}}^{\beta}(B^{\alpha}))^{-1}\mathcal{B}_k$ is a basis of $\mathbb{P}_k(T)$ dual to $\mathcal{L}$.
\end{lemma}
\begin{proof}
Let $\{\texttt{b}^{\alpha}, \alpha \in \mathbb{T}_k^d\}$ be the basis of $\mathbb{P}_k(T)'$ dual to $\mathcal{B}_k$. 
Assume $\texttt{b}^{\gamma} = \sum_{\alpha \in \mathbb{T}_k^d} C_{\gamma, \alpha} \texttt{l}^{\alpha}$ for $\gamma \in \mathbb{T}_k^d$, where $C_{\gamma, \alpha} \in \mathbb{R}$. By the duality property, we have:
\[
\sum_{\alpha \in \mathbb{T}_k^d} C_{\gamma, \alpha} \texttt{l}^{\alpha}(B^{\beta}) = \delta_{\gamma, \beta}, \quad \forall~ \gamma, \beta \in \mathbb{T}_k^d.
\]
This implies that $(\texttt{l}^{\alpha}(B^{\beta}))^{-1}\mathcal{L}$ is dual to $\mathcal{B}_k$.

For the second part, let $\{\texttt{G}^{\gamma}, \gamma \in \mathbb{T}_k^d\}$ be the basis of $\mathbb{P}_k(T)$ dual to $\mathcal{L}$. Assume $\texttt{G}^{\gamma} = \sum_{\beta \in \mathbb{T}_k^d} C_{\gamma, \beta} B^{\beta}$ for $\gamma \in \mathbb{T}_k^d$, where $C_{\gamma, \beta} \in \mathbb{R}$. By the duality property, we get:
\[
\sum_{\beta \in \mathbb{T}_k^d} C_{\gamma, \beta} \texttt{l}^{\alpha}(B^{\beta}) = \delta_{\gamma, \alpha}, \quad \forall~ \gamma, \alpha \in \mathbb{T}_k^d.
\]
Thus, $((\texttt{l}^{\alpha}(B^{\beta}))^{\intercal})^{-1} \mathcal{B}_k$ is dual to $\mathcal{L}$.
\end{proof}

For instance, we can use DoFs such as $\{\int_f u \lambda^{\alpha_f}\dd s\}$ for the Lagrange finite elements. The corresponding DoF-Basis matrix $(\texttt{l}^{\alpha}(B^{\beta}))_{\alpha,\beta \in \mathbb{T}_k^d}$ is block lower triangular~\cite{Chen;Huang:2021Geometric}. For lattice points at vertices, i.e., $\alpha \in \mathbb{T}_k^{0}(\texttt{v})$, the DoF $u(\texttt{v})$ corresponds to the function value at the vertex. For sub-simplices of dimension $\ell \geq 1$, inverting a Gram matrix to obtain the basis dual to the Bernstein basis.

\begin{remark}\rm
A popular set of DoFs is the function value at all lattice points, i.e., $\texttt{l}^{\alpha}(v) = v(x_{\alpha})$. The corresponding dual basis is known as the Lagrange basis~\cite{nicolaides1972class}:
$$
L_{\alpha}(x) = 
\frac{1}{\alpha!} 
\prod_{i=0}^{d}\prod_{j=0}^{\alpha_i - 1} (k\lambda_i(x) - j), \quad 
\alpha \in \mathbb{T}_k^d.
$$
It is straightforward to verify the duality between the basis and the DoFs:
$$
\texttt{l}^{\alpha}(L_{\beta}) = L_{\beta}(x_{\alpha}) = \delta_{\alpha, \beta}, \quad \alpha, \beta \in \mathbb{T}_k^d.
$$
However, no simple formula for $\nabla^r L_{\alpha}$ is available. Therefore, we will retain the Bernstein basis function and modify the dual basis accordingly.
\end{remark}

\subsection{Dual Basis of Bernstein Basis}
The dual basis of $\mathcal{B}_k$ is a set of linear functionals $\mathcal{B}_k' := \{\texttt{b}^{\alpha}, \alpha \in \mathbb{T}_k^d\} \subset \mathbb{P}_k(T)'$ such that:
\begin{equation}\label{eq:bdual}
\texttt{b}^{\alpha}(B^{\beta}) := \langle \texttt{b}^{\alpha}, B^{\beta} \rangle = \delta_{\alpha, \beta},
\end{equation}
where $\langle \cdot, \cdot \rangle$ is the duality pairing. 

Similarly, the dual basis of $\mathcal{B}_k(f)$ on the sub-simplex $f$ is defined as:
\[
\begin{aligned}
\mathcal{B}_k(f)' &= \{\texttt{b}^{\alpha_f}, \alpha_f \in \mathbb{T}_k^{\ell}(f)\} \in \mathbb{P}_k(f)' \quad \text{such that} \\
\quad
\texttt{b}^{\alpha_f}(B^{\beta_f}) &:= \langle \texttt{b}^{\alpha_f}, B^{\beta_f} \rangle = \delta_{\alpha_f, \beta_f}, \quad \forall~\alpha_f, \beta_f \in \mathbb{T}_k^\ell(f).
\end{aligned}
\]
The functional $\texttt{b}^{\alpha_f}$ can be extended to $\mathbb P_k(T)'$ by the natural restriction of function value, i.e.
$$
\texttt{b}^{\alpha_f}(u) = \langle \texttt{b}^{\alpha_f}, u\mid _f \rangle, \quad u\in \mathbb P_k(T). 
$$
When $\ell = 0$, i.e., at a vertex $\texttt{v}$,     
$$
\texttt{b}^{\alpha_{\texttt{v}}}(u) =    \langle \ttt{b}^{\alpha_{\ttt{v}}}, u\mid_\ttt{v} \rangle = u(\ttt{v}), \quad \forall~u \in \mathbb{P}_{|\alpha_\ttt{v}|}(T).
$$

By Lemma \ref{lm:constructdualbases}, we can find an explicit formula for $\texttt{b}^{\alpha}$ by inverting the DoF-Basis matrix for some DoFs. For example, for a given $f\in \Delta_{\ell}(T), \ell \geq 1$, considering the DoFs
$$
\textnormal{\texttt{l}}^{\alpha}(\cdot):= \int_{f} B^{\alpha_f}(\cdot) \dd s,
$$
then the dual basis $(\texttt{b}^{\alpha_f}) = ( \int_f B^{\alpha_f}B^{\beta_f}\dd s)^{-1}(\textnormal{\texttt{l}}^{\alpha})$. By \eqref{eq:int-bernstein}, we have an explicit formula on the Gram matrix 
$$
\int_f B^{\alpha_f}B^{\beta_f}\dd s = \frac{(\alpha_f+\beta_f)!k!k!\ell !}{\alpha_f!\beta_f!(2k+\ell)!}|f|, \quad \alpha_f, \beta_f\in \mathbb T^{\ell}_k,
$$
but it is hard to write its inverse.

We do not necessarily need to form the explicit formula for $\texttt{b}^{\alpha}$; the important property is the duality relation in \eqref{eq:bdual}, which guarantees the functional behavior required for our finite element construction.

\subsection{Normal Basis}
For a sub-simplex \( f \in \Delta_{\ell}(T) \), choose \( \ell \) linearly independent (not necessarily orthogonal) tangential vectors \( \{\boldsymbol{t}^1_f, \dots, \boldsymbol{t}^{\ell}_f\} \) for \( f \) and \( d - \ell \) linearly independent (not necessarily orthogonal) normal vectors \( \{\boldsymbol{n}^1_f, \dots, \boldsymbol{n}^{d-\ell}_f\} \) for \( f \). The set of \( d \) vectors \( \{\boldsymbol{t}^1_f, \dots, \boldsymbol{t}^{\ell}_f, \boldsymbol{n}^1_f, \dots, \boldsymbol{n}^{d-\ell}_f\} \) forms a basis for \( \mathbb{R}^d \). The tangent and normal planes of \( f \) are defined as:
\[
\mathscr{T}^f := \operatorname{span} \{\boldsymbol{t}^i_f \mid i = 1, \ldots, \ell\}, \quad
\mathscr{N}^f := \operatorname{span} \{\boldsymbol{n}^i_f \mid i = 1, \ldots, d - \ell\}.
\]

We now introduce two bases for its normal plane \( \mathscr{N}^f \). Recall that \( F_i \) represents the \((d - 1)\)-dimensional face opposite the \( i \)-th vertex. Hence, \( f \subseteq F_i \) for \( i \in f^* \). The vector \( \nabla \lambda_i \in \mathbb{R}^d \) is normal to \( F_i \) and thus normal to \( f \subseteq F_i \) for \( i \in f^* \). Let its projection onto \( \mathscr{T}^f \) be denoted as \( \nabla_f \lambda_i \), which is also called the surface or tangential gradient.

For $f \in \Delta_{\ell}(T)$, $0\leq \ell \leq d-1$, and for $i \in f^*$, let $f \cup \{i\}$ represent the $(\ell+1)$-dimensional face in $\Delta_{\ell+1}(T)$ with vertices $\{i, f(0), \dots, f(\ell)\}$. The tangential gradient $\nabla_{f \cup \{i\}} \lambda_i$ is normal to $f$ but tangential to $f \cup \{i\}$. 

\begin{figure}[htbp]
\subfigure{
\begin{minipage}[t]{0.475\linewidth}
\centering
\includegraphics[width=6cm]{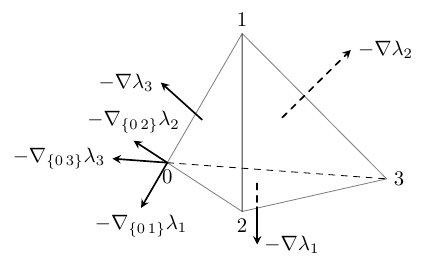}
\end{minipage}}
\quad
\subfigure
{
\begin{minipage}[t]{0.44\linewidth}
\centering
\includegraphics*[width=5cm]{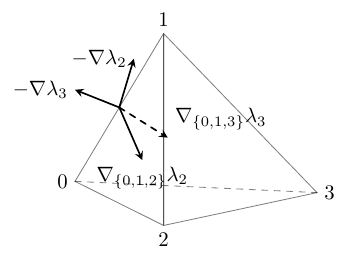}
\end{minipage}}
\caption{Face normal basis and tangential-normal basis at a vertex and an edge.}
\label{fig:normalbasis}
\end{figure}

We claim that these two bases of $\mathscr{N}^f$ are dual to each other with appropriate scaling:
\begin{align*}
&\text{Face normal basis: } &\{ \nabla \lambda_i \mid i \in f^* \}, \\
&\text{Tangential-normal basis: } &\{ \nabla_{f \cup \{i\}} \lambda_i \mid i \in f^* \}.
\end{align*}

\begin{lemma}\label{lem:normaldual}
For $f \in \Delta_{\ell}(T)$, the rescaled tangential-normal basis 
$$
\left\{ \frac{\nabla_{f \cup \{i\}} \lambda_i}{|\nabla_{f \cup \{i\}} \lambda_i |^2} \mid i \in f^* \right\}
$$ 
of $\mathscr{N}^f$ is dual to the face normal basis $\{\nabla \lambda_i \mid i \in f^*\}$.
\end{lemma}
\begin{proof}
Clearly, $\nabla_{f \cup \{i\}} \lambda_i, \nabla\lambda_i \in \mathscr{N}^f$ for $i \in f^*$. As the orthogonal projection of $\nabla \lambda_i$ to the face $f \cup \{i\}$, 
$$
\nabla_{f \cup \{i\}} \lambda_i \cdot \nabla\lambda_i = |\nabla_{f \cup \{i\}} \lambda_i |^2.
$$
It suffices to prove
\[
\nabla_{f \cup \{i\}} \lambda_i \cdot \nabla\lambda_j = 0 \quad \text{for } i, j \in f^*, \, i \neq j,
\]
which follows from the fact that $f \cup \{i\} \subseteq F_j$ and $\nabla_{f \cup \{i\}} \lambda_i \in \mathscr{T}^{f \cup \{i\}}$. 
\end{proof}

\subsection{Integral Form DoFs}
We recall the degrees of freedom (DoFs) defined in \cite{HuLinWu2024,chen_geometric_2021,Chen;Huang:2022FEMcomplex3D}. For each sub-simplex $f \in \Delta_{\ell}(T)$, we choose a normal basis $n_f = \{\boldsymbol{n}_f^i, i = 1, \dots, d - \ell\}$ for its normal plane $\mathscr{N}^f$ to define the normal derivatives 
$\frac{\partial^{|\beta|} u}{\partial n_f^{\beta}}, \beta \in \mathbb{N}^{1:d-\ell}.$

\begin{theorem}\label{th:applocalPrCm}
Given an integer $m \geq 0$, let $\bs{r} = (r_0, r_1, \dots, r_d)$ satisfy
$$
r_d = 0, \quad r_{d-1} = m, \quad r_{\ell} \geq 2r_{\ell+1} \quad \textrm{for } \ell = d - 2, \dots, 0.
$$
Assume $k \geq 2r_0 + 1 \geq 2^d m + 1$. Then the shape function space $\mathbb{P}_k(T)$ is uniquely determined by the following DoFs:
\begin{align*}
D^{\alpha} u (\textnormal{\texttt{v}}) & \quad \alpha \in \mathbb{N}^{1:d}, \, |\alpha| \leq r_0, \, \textnormal{\texttt{v}} \in \Delta_0(T), \\
\int_f \frac{\partial^{|\beta|} u}{\partial n_f^{\beta}}  \, \lambda_f^{\alpha_f} \, \mathrm{d}s & \quad \alpha \in S_{\ell}(f), \, |\alpha_f| = k - s, \, \beta \in \mathbb{N}^{1:d-\ell}, \, |\beta| = s, \\
& \quad f \in \Delta_{\ell}(T), \, \ell = 1, \dots, d - 1, \, s = 0, \dots, r_{\ell}, \notag \\
\int_T u \, \lambda^{\alpha} \, \mathrm{d}x & \quad \alpha \in S_d(T).
\end{align*}
\end{theorem}

We have shown the uni-solvence in~\cite{chen_geometric_2021,Chen;Huang:2022FEMcomplex3D} by demonstrating that the DoF-Basis matrix is block lower triangular. To find a basis dual to this set of DoFs, we need to invert this lower triangular matrix, which will involve a Gram matrix $\Big (\int_f \lambda_f^{\alpha_f}\lambda_f^{\beta_f} \dd s \Big )$.

\subsection{Modified Degree of Freedoms}
\label{sec:modifiedDoF}

For a sub-simplex $f \in \Delta_{\ell}(T)$, we choose the rescaled tangential-normal basis 
$$
n_f = \left\{ \frac{\nabla_{f \cup \{i\}} \lambda_i}{|\nabla_{f \cup \{i\}} \lambda_i|^2} \mid i \in f^* \right\}
$$ 
of $\mathscr{N}^f$. 
The modified degrees of freedom (DoFs) are defined as:
\begin{equation}\label{eq:bDoF}
\begin{aligned}
&\quad \texttt{L}_f^{\alpha}(u) := \langle \texttt{b}^{\alpha_f}, \frac{\partial^{|\alpha_{f^*}|}}{\partial n_f^{\alpha_{f^*}}} u \mid _f \rangle, \\
&  \alpha = E_f(\alpha_f) + E_{f^*}(\alpha_{f^*}) \in S_{\ell}(f), \alpha_f\in \mathbb T^{\ell}_{k-s}, \alpha_{f^*}\in \mathbb T^{d-\ell-1}_{s}, \\
&  s = 0,1, \ldots, r_{\ell}, f\in \Delta_{\ell}(T), \ell = 0, 1, \ldots, d.
\end{aligned}
\end{equation}

\begin{theorem}
The DoF-Basis matrix $(\texttt{L}_f^{\alpha}(B^{\beta}))$ with $\texttt{L}_f^{\alpha}$ defined by \eqref{eq:bDoF} is block lower triangular and invertible. Furthermore, we have the formulae
$$
\begin{aligned}
D_{\alpha, \beta} &:= \texttt{L}_f^{\alpha}(B^{\beta}) = \langle \texttt{b}^{\alpha_f},
\frac{\partial^{|\alpha_{f^*}|}}{\partial n_f^{\alpha_{f^*}}} B^{\beta} \mid_f \rangle =
\langle \texttt{b}^{\alpha_f}, \nabla^{|\alpha_{f^*}|} B^{\beta} \mid_f :
n_f^{\alpha_{f^*}} \rangle, \quad \alpha, \beta \in \mathbb{T}_k^d(T), \\
D_{\alpha, \beta} &= \frac{k!}{(k - s)!} \delta_{\alpha, \beta}, \quad \alpha, \beta \in S_{\ell}(f) \cap L(f, s), \quad s = 0, 1, \ldots, r_{\ell},\ \ell = 0, 1, \ldots, d.
\end{aligned}
$$
\end{theorem}
\begin{proof}
For $\alpha \in S_{\ell}(f)$ and $\beta \in S_m(g)$ with $\ell \leq m$, $f \in \Delta_{\ell}(T)$, $g \in \Delta_{m}(T)$, and $f \neq g$, we have ${\rm dist}(\beta, f) > r_{\ell} \geq {\rm dist}(\alpha, f) = |\alpha_{f^*}|$. Applying Lemma~\ref{lm:derivative} to get $(\frac{\partial^{|\alpha_{f^*}|}}{\partial n_f^{\alpha_{f^*}}} B^{\beta})|_f = 0$ and thus $\texttt{L}_f^{\alpha}(B^{\beta}) = 0$. 

Thus, the DoF-Basis matrix is block lower triangular when sorted by the dimension of the decomposition of the lattice. The structure of the matrix can be illustrated as:
$$
\renewcommand{\arraystretch}{1.35}
\begin{array}{cc}
\begin{array}{c}
 \; \texttt{L}_f^{\alpha} \backslash \;   B^{\beta} 
\end{array}
 &  
\begin{array}{ccccc}
\quad S_0 \qquad\qquad\;\;\;& \!\!\! S_1\;\;\;\;\; \quad\quad& \;\ldots\quad\quad\;\;\;	&  S_{d-1}\qquad\quad & \quad\, S_d\quad\quad
\end{array}
\medskip
\\
\begin{array}{c}
S_0 \\ S_1 \\ \vdots \\ S_{d-1} \\ S_d
\end{array} 
& \left(
\begin{array}{>{\hfil$}m{1.5cm}<{$\hfil}|>{\hfil$}m{1.5cm}<{$\hfil}|>{\hfil$}m{1.5cm}<{$\hfil}|>{\hfil$}m{1.5cm}<{$\hfil}|>{\hfil$}m{1.5cm}<{$\hfil}}
\square & 0 & \cdots	& 0 & 0 \\
\hline
\square & \square & \cdots	& 0 & 0 \\
\hline
\vdots & \vdots & \ddots	& \vdots & \vdots \\
\hline
\square & \square & \cdots	& \square & 0 \\
\hline
\square & \square & \cdots	& \square& \square 
\end{array}
\right)
\end{array}.
$$

We now consider one diagonal block. Let $\alpha, \beta \in S_{\ell}(f)$, where $f \in \Delta_{\ell}(T)$. When $|\beta_{f^*}| > |\alpha_{f^*}|$, we apply Lemma~\ref{lm:derivative} again to obtain $\texttt{L}_f^{\alpha}(B^{\beta}) = 0$. 
Thus, sorting by the distance to $f$, i.e. for $s=0,1, \ldots, r_{\ell}$, the DoF-Basis sub-matrix $(\texttt{L}_f^{\alpha}(B^{\beta}))_{\alpha, \beta \in S_{\ell}(f)}$ is also block lower triangular. 

Next, we consider the diagonal block $(\texttt{L}_f^{\alpha}(B^{\beta}))_{\alpha, \beta \in S_{\ell}(f) \cap L(f, s)}$, for a fixed $s = 0, 1, \ldots, r_{\ell}$, where $|\beta_{f^*}| = |\alpha_{f^*}| = s$ and $|\beta_{f}| = |\alpha_{f}| = k-s$. By \eqref{eq:bern_grad}, we have:
\[
\nabla^{s} B^{\beta} = \sum_{\tilde{\alpha} \in \mathbb{T}_s^d, \tilde{\alpha} \leq \beta} \frac{k! s!}{(k-s)! \tilde{\alpha}!} \operatorname{sym}((\nabla \lambda)^{\otimes \tilde{\alpha}}) B^{\beta - \tilde{\alpha}}.
\]
Noting that
$$ 
\frac{\partial^{s}}{\partial n_f^{\alpha_{f^*}}} B^{\beta} = \nabla^{s} B^{\beta} : \left( (\boldsymbol{n}_f^1)^{\otimes (\alpha_{f^*})_1} \otimes \cdots \otimes (\boldsymbol{n}_f^{d-\ell})^{\otimes (\alpha_{f^*})_{d-\ell}} \right),
$$
we obtain:
\begin{align*}
\langle \texttt{b}^{\alpha_f}, \frac{\partial^{s}}{\partial n_f^{\alpha_{f^*}}} B^{\beta} \mid _f\rangle
&= \frac{k! s!}{(k-s)! (\beta - E(\alpha_f))!} \operatorname{sym}((\nabla \lambda)^{\otimes (\beta - E(\alpha_f))}) : n_f^{\alpha_{f^*}} \\
&= \frac{k!}{(k-s)!}\delta_{\alpha, \beta}.
\end{align*}

Hence, the DoF-Basis sub-matrix $(\texttt{L}_f^{\alpha}(B^{\beta}))_{\alpha, \beta \in S_{\ell}(f)}$ corresponding to $S_{\ell}(f)$ is block lower triangular, with each diagonal block being a positively rescaled identity matrix. Thus, the DoF-Basis matrix is block lower triangular and invertible.
\end{proof}

By Lemma~\ref{lm:constructdualbases}, we can invert the transpose of the DoF-Basis matrix to find the local basis (as a linear combination of Bernstein basis) that is dual to the DoFs in \eqref{eq:bDoF}.  
The transpose $(D_{\beta, \alpha})$ is upper triangular, so we construct the dual basis backward on $S_\ell$ for $\ell = d, d-1, \ldots, 0$.  
Similarly, for each $S_\ell(f)$, we proceed by descending distance $s = r_\ell, r_\ell - 1, \ldots, 0$.

For the interior DoFs on $S_d$, the submatrix is the identity:
$$
D_{\alpha, \beta} = \delta_{\alpha, \beta}, \quad \alpha, \beta \in S_d(T).
$$
No modification is needed. In general, basis functions computed later may depend on those computed earlier.

\section{Global Frame, Degrees of Freedom, and Basis}\label{sec:global}
The global DoFs take the same form as the local ones. However, to enforce the required continuity, the global DoFs are uniquely determined by the sub-simplex $f$, not the element containing $f$. This is achieved by a unique labeling of the reference lattice points on $f$ and the choice of a global normal basis.

\subsection{Reference Lattice Decomposition}
The lattice decomposition \eqref{eq:appdecT} is element-wise, meaning that $S_{\ell}(f)$ depends on the element $T$ containing $f$. To overcome this dependency, we introduce a reference set $\hat{S}_{\ell}([f])$ which depends only on $f$, and we note that any $S_{\ell}(f)$ can be viewed as a mapping of $\hat{S}_{\ell}([f])$.

Recall that the restriction operator $R_f: \mathbb{T}_k^d \to \mathbb{T}_s^{\ell}$ is defined in \eqref{eq:restriction}. For a sub-simplex $f \in \Delta_{\ell}(\mathcal{T}_h)$, as an abstract simplex, it may appear in different simplices, say $T_1$ and $T_2$, with different local orderings. For example, a face $f = \{2, 10, 7\} \in \Delta_2(T_1)$ might appear as $f = \{10, 2, 7\}$ in $\Delta_2(T_2)$. The ordering of the vertices of $f$ affects the restriction $\alpha_f$, and to avoid ambiguity, we fix the ordering of abstract simplices by the ascending order of the global index of vertices of the triangulation $\mathcal T_h$ and denote by $[f]$. In the example above, in all simplices containing $f$, it will have the same ordering $[f] = \{2, 7, 10\}$.

Note that the ascending ordering may not always induce the positive orientation of the $d$-simplex $T$ for a given geometric realization. For detailed discussions on indexing, ordering, and orientation via the \texttt{sc} and \texttt{sc3} documentation in $i$\texttt{FEM}, we refer to~\cite{Chen2008ifem}.

With such an ordering, the image of $R_f$ is independent of the element $T$ containing $f$. Namely, $R_f(S_{\ell}(f))$ is uniquely determined for $S_{\ell}(f)\subset \mathbb{T}_k^d(T)$ with $f \in \Delta_{\ell}(T)$ and $T\in \mathcal T_h$. Now, we can introduce the reference lattice set
\begin{equation}\label{eq:hatS}
\hat{S}_{\ell}([f]):= \{(\alpha_f, \gamma) \mid \alpha_f \in R_f(S_{\ell}(f)), \gamma \in \mathbb{T}_{k - |\alpha_{f}|}^{d-\ell - 1} \},
\end{equation}
where the ascending order $[f]$ is used to emphasize the induced orientation depends only on $f$ not on the element $T$ containing $f$. 

\begin{lemma}
For sub-simplex $f \in \Delta_{\ell}(\mathcal{T}_h)$, let $S_{\ell}(f)\subset \mathbb{T}_k^d(T)$ be the subset in decomposition \eqref{eq:appdecT} for some $T\in \mathcal{T}_h$ containing $f$. Then
$$
R: S_{\ell}(f) \to \hat{S}_{\ell}([f]), \quad R(\alpha) = (\alpha_f, \alpha_{f^*}),
$$
is a one-to-one map.
\end{lemma}
\begin{proof}
Obviously, $R$ is injective. We then show that it is surjective. Take $(\alpha_f, \gamma)\in \hat{S}_{\ell}([f])$, let $\alpha\in S_{\ell}(f)$ such that $R_f(\alpha) = \alpha_f$. Define
$$
\bar{\alpha} = E_f(\alpha_f) + E_{f^*}(\gamma) \in \mathbb{T}_k^d(T).
$$

We claim that $\bar{\alpha} \in S_{\ell}(f)$. First of all, $R_f(\alpha) = R_f(\bar{\alpha})$. Take any $e \in \Delta(f)$. Since $e \subseteq f$, we have $f^* \subseteq e^*$, and
$$
e^* \setminus f^* = e^* \cap f = f \cap e^* = f \setminus e.
$$
Thus,
$$
\begin{aligned}
\dist(\bar{\alpha}, e) &= |R_{e^*}(\bar{\alpha})| = |R_{f^*}(\bar{\alpha})| + |R_{f \setminus e}(\bar{\alpha})| = |\gamma| + |R_{f \setminus e}(\alpha)| \\
&= |\gamma| + |R_f(\alpha)| - |R_e(\alpha)| = k - |R_e(\alpha)| = \dist(\alpha, e).
\end{aligned}
$$
As $\alpha\in S_{\ell}(f)$, all distance conditions are satisfied, and consequently, $\bar{\alpha} \in S_{\ell}(f)$.
\end{proof}

Using the reference lattice points, we will have a direct union
\[
\mathbb S_{k,\bs r}^d:= \Oplus_{\ell=0}^d\Oplus_{f\in \Delta_{\ell}(\mathcal T_h)}\hat{S}_{\ell}([f]),
\]
which is a generalization of lattice decomposition \eqref{eq:Xsum} for $\bs r = \bs 0$ to smoothness vector $\bs{r} = (r_0, r_1, \dots, r_d)$.

We will assign a labeling of $\mathbb S_{k,\bs r}^d$, i.e., each lattice point in $\mathbb S_{k,\bs r}^d$ will have a unique index called the global index of a reference lattice point. The lattice decomposition \eqref{eq:appdecT} is element-wise and each lattice point in $\mathbb T_k^d$ will have a local index. During the assembling process, it is unavoidable to figure out the mapping between the local index and the global index of a lattice point.

To describe this mapping, we explicitly include the notation of the standard abstract complex $\texttt{S}_d = \{0, 1, \ldots, d\}$, an abstract complex $\TT = \{\TT(0), \TT(1), \ldots, \TT(d)\}$ with $\TT(i)\in V = \{1, 2, \ldots, N\}$, and $T\in \mathcal T_h$ as a geometric realization of $\TT$. We write
\begin{align}
\label{eq:localdec} \mathbb{T}_k^d(\texttt{S}_d) &= \Oplus_{\ell=0}^d \Oplus_{f \in \Delta_\ell(\texttt{S}_d)} S_\ell(f(\texttt{S}_d)), \\
\notag
\mathbb{T}_k^d(\texttt{T}) &= \Oplus_{\ell=0}^d \Oplus_{f \in \Delta_\ell(\TT)} S_\ell(f(\texttt{T})) \stackrel{R}{\longrightarrow}  \Oplus_{\ell=0}^d \Oplus_{f(\TT) \in \Delta_\ell(\mathcal T_h)} \hat S_\ell([f(\texttt{T})]).
\end{align}
We will assign a labeling of lattice points in $S_\ell(f(\texttt{S}_d))$ based on the lattice decomposition \eqref{eq:localdec}, which is called a local indexing. The face $f(\texttt{S}_d)\subset \{0,1,\ldots, d\}$ and the local face $f(\texttt{T})\subset \{\TT(0), \TT(1), \ldots \TT(d)\}$ will have vertices with a global index. Therefore, $f(\texttt{T}) \in \Delta_{\ell}(\mathcal T_h)$ and $S_\ell(f(\texttt{T}))$ can be mapped to $\hat S_\ell([f(\texttt{T})])$, which gives a local to global index map. The roadmap is summarized below
$$
S_\ell(f(\texttt{S}_d)) \to S_\ell(f(\texttt{T})) \to \hat S_\ell([f(\texttt{T})]).
$$
If $[f(\texttt{T})]$ is not used in $\hat S_\ell$, one needs to find a permutation between local and global faces~\cite{ChenChenHuangWei2023}.
 

%
%

\subsection{Algorithm to Find the Lattice Decomposition}
%
To begin with, we generate the lattice points $\mathbb{T}^d_k$ via the identity:
$$
\mathbb{T}^d_k = \Oplus_{i=0}^k \{[i, \alpha] \mid \alpha \in \mathbb{T}^{d-1}_{k-i}\}.
$$
Thus, $\mathbb{T}^d_k$ can be computed using the recursive algorithm shown in Algorithm~\ref{alg:lattice}.

\begin{algorithm}
\caption{Generate Lattice Points (GLP$(k,d)$)}
\label{alg:lattice}
\begin{algorithmic}[1]
\State \textbf{Input:} $k, d$ \Comment{degree and dimension}
\State \textbf{Output:} $\mathbb{T}_k^d$ \Comment{lattice points}
\Function{GLP}{$k, d$}
    \If{$d = 0$}
        \State \Return $\lstinline{[[k]]}$
    \Else
        \State Initialize empty list $\mathbb{T}_k^d$
        \For{$i$ from $k$ down to $0$}
            \State $\mathbb{T}_{k-i}^{d-1} \gets \Call{GLP}{k-i, d-1}$
            \For{each $\beta$ in $\mathbb{T}_{k-i}^{d-1}$}
                \State Append $[i, \beta]$ to $\mathbb{T}_k^d$ 
            \EndFor
        \EndFor
        \State \Return $\mathbb{T}_k^d$ 
    \EndIf
\EndFunction
\end{algorithmic}
\end{algorithm}

The set $D(f, r)$ is a subset of $\mathbb{T}^d_k$, obtained by filtering those
lattice points whose $|\alpha_{f^*}|$ is less than $r$.  
Then, $S_{\ell}(f)$ is derived from $D(f, r)$ through a set operation, 
and $\hat{S}_{\ell}([f])$ is further constructed by modifying $S_{\ell}(f)$, 
as defined in~\eqref{eq:hatS}.


\subsection{Global DoFs}
For a sub-simplex $f \in \Delta_{\ell}(\mathcal{T}_h)$, we shall choose a global basis $N_f := \{\bs{N}_f^1, \ldots, \bs{N}_f^{d-\ell}\}$ for the normal plane $\mathscr{N}^f$, i.e., depending on $f$ rather than the element containing $f$.

The global DoFs of the $C^m$ finite element  
on $\mathcal{T}_h$ are defined by:
\begin{equation}
    \label{eq:globalDoF}
\texttt{G}^{\alpha}(u) := \langle \texttt{b}^{\alpha_f}, \frac{\partial^{|\gamma|}}{\partial
    N_f^{\gamma}}u \mid_f \rangle, \quad \alpha = (\alpha_f, \gamma) \in
    \hat{S}_{\ell}([f]),\ f\in \Delta_{\ell}(\mathcal{T}_h),\ \ell = 0, 1, \ldots, d,
\end{equation}
which is independent of the element containing $f$.

\begin{theorem}
The DoF \eqref{eq:globalDoF} will define a $C^m$-conforming finite element space $\mathcal S_{k,\bs r}^d(\mathcal T_h)$. 
\end{theorem}
\begin{proof}
Both the integral DoFs $\{\texttt{i}^{\alpha_f}(\cdot) := \int_f \lambda^{\alpha_f}(\cdot) \dd s \mid \alpha_f \in R_f(S_{\ell}(f)\cap L(f,s))\}$ and $\{\langle \texttt{b}^{\alpha_f},\cdot \rangle \mid \alpha_f \in R_f(S_{\ell}(f)\cap L(f,s))\}$ are bases of $\mathbb P_{k-s}'(R_f(S_{\ell}(f)\cap L(f,s)))$. Therefore the DoFs
$$
\int_f \lambda^{\alpha_f}\frac{\partial^{|\gamma|}}{\partial
    N_f^{\gamma}}u \dd s
$$
and $\texttt{G}^{\alpha}(u)$ can be expressed in terms of each other.
By Theorem A.8 in \cite{Chen;Huang:2022FEMcomplex3D}, we conclude the function is in $C^m(\Omega)$. 
\end{proof}
As a byproduct, we have the following dimension formula for a $C^m(\Omega)$-conforming finite element space $\mathcal S_{k,\bs r}^d(\mathcal T_h)$
$$
\dim \mathcal S_{k,\bs r}^d(\mathcal T_h) = \sum_{\ell = 0}^d |\Delta_{\ell}(\mathcal T_h)| |\hat{S}_{\ell}([f])|.
$$
The cardinality $|\hat{S}_{\ell}([f])|$ is hard to have an explicit formula, but we have presented a numerical method to find the lattice decomposition in \S 5.2 and $\dim \mathcal S_{k,\bs r}^d(\mathcal T_h)$ is computable. 

\subsection{Transformation From Local to Global DoFs}
\label{sec:local2global}
The normal derivatives for the local normal basis and the global normal basis can be related by the change of variables. Recall that $\{\texttt{L}^{\alpha}\}_{\alpha \in \mathbb{T}_k^d}$ is the local DoFs as in \eqref{eq:bDoF}, and $\{\texttt{G}^{\alpha}\}_{\alpha \in \mathbb{T}_k^d}$ is the global DoFs as in \eqref{eq:globalDoF}. Then there exists a matrix $\boldsymbol{T} = (T_{\alpha,\beta})$ such that:
$$
(\texttt{L}^{\alpha}) =
(T_{\alpha,\beta}) (\texttt{G}^{\beta}), \quad \alpha, \beta \in \mathbb{T}_k^d.
$$
Consequently, from
$$
\langle (\texttt{L}^{\alpha}), (\psi_{\alpha})\rangle = \langle (T_{\alpha,\beta}) (\texttt{G}^{\beta}), (\psi_{\alpha}) \rangle = \langle  (\texttt{G}^{\beta}), (T_{\alpha,\beta})^{\intercal}(\psi_{\alpha}) \rangle = \langle (\texttt{G}^{\beta}), (T_{\beta, \alpha}) (\psi_{\alpha}) \rangle,
$$
we conclude that 
$\{ T_{\beta,\alpha} \psi_{\alpha} \}$
is the set of basis functions dual to
$\{\texttt{G}^{\beta}\}$, where $\{\psi_{\alpha}\}$ is the set
of basis functions dual to $\{\texttt{L}^{\alpha}\}$ and has been computed locally in Section~\ref{sec:local_frame}.


Now, we present the computation of the transformation matrix $\{ T_{\alpha, \beta} \}$. 
Let $f \in \Delta_{\ell}(\mathcal{T}_h)$, and suppose the dual normal basis of  
$N_f = \{\bs{N}_f^1, \ldots, \bs{N}_f^{d-\ell}\}$  
is denoted by $\hat{N}_f := \{\hat{\boldsymbol{N}}_f^1, \ldots, \hat{\boldsymbol{N}}_f^{d-\ell}\}$.  
Let $n_f$ denote the local normal basis of $f$ in some element $T\in \mathcal T_h$ containing $f$.  

For $\alpha \in \mathbb{T}_k^{d}$, denote by $s = |\alpha_{f^*}| = \dist(\alpha, f) $.  
According to Lemma~\ref{lm:Sbasis} and \eqref{eq:tensordual}, the symmetric tensor  
$\sym(n_f^{\otimes\alpha_{f^*}})$ can be written as a linear combination of  
$\{\sym(N_f^{\otimes\gamma})\}_{\gamma \in \mathbb{T}_s^{d-\ell-1}}$ as follows:
$$
\sym(n_f^{\otimes\alpha_{f^*}}) = \sum_{\gamma \in \mathbb{T}_s^{d-\ell-1}} 
\tilde{T}_{\alpha_{f^*}, \gamma} \, \sym(N_f^{\otimes\gamma}),
$$
where
$$
\tilde{T}_{\alpha_{f^*}, \gamma} = \frac{s!}{\gamma!} \, \sym(n_f^{\otimes\alpha_{f^*}}) : 
\sym(\hat{N}_f^{\otimes\gamma}).
$$
This implies that
$$
\frac{\partial^{s}}{\partial n_f^{\alpha_{f^*}}}u = 
\nabla^{s}u : n_f^{\otimes\alpha_{f^*}} 
= \sum_{\gamma \in \mathbb{T}_r^{d-\ell-1}}
\tilde{T}_{\alpha_{f^*}, \gamma} \nabla^{s}u : \sym(N_f^{\otimes\gamma})
= \sum_{\gamma \in \mathbb{T}_r^{d-\ell-1}}
\tilde{T}_{\alpha_{f^*}, \gamma} \frac{\partial^{s}}{\partial
N_f^{\gamma}}u.
$$
Let $\beta \in \mathbb{T}_k^{d}$ be such that $\beta_f = \alpha_f$ and $\beta_{f^*} = \gamma$.  
Then the formula above gives:
$$
\langle \texttt{b}^{\alpha_f}, \, \frac{\partial^{s}}{\partial n_f^{\alpha_{f^*}}}u \mid_f \rangle = 
\sum_{\beta \in \mathbb{T}_k^{d}, \, \beta_f = \alpha_f} 
\tilde{T}_{\alpha_{f^*}, \beta_{f^*}} \, 
\langle \texttt{b}^{\beta_f}, \, \frac{\partial^{|\beta_{f^*}|}}{\partial N_f^{\beta_{f^*}}}u \mid_f\rangle.
$$

Hence, we obtain the explicit form of the transformation matrix:
\begin{equation} \label{eq:transformmatrix}
    T_{\alpha,\beta} = 
    \begin{cases}
        \displaystyle \frac{s!}{\beta_{f^*}!} \, \sym(n_f^{\otimes\alpha_{f^*}}) : 
        \sym(\hat{N}_f^{\otimes\beta_{f^*}}), 
        & \text{if } \alpha, \beta \in S_{\ell}(f)\cap L(f, s), \\
        \quad 0, & \text{otherwise.}
    \end{cases}
\end{equation}

We summarize the discussion in the following theorem and provide an example for the lowest-order $C^1$ finite element in two dimensions in Appendix~\ref{appdx:C12d}.
\begin{theorem}\label{lm:GlobalBasis}
Let $\{\psi_{\alpha}\}_{\alpha \in \mathbb{T}_k^d}$ be the set of local basis functions dual to the degrees of freedom $\{\texttt{L}^{\alpha}\}_{\alpha \in \mathbb{T}_k^d}$ defined in~\eqref{eq:bDoF}, constructed using local normal frames. Then, the corresponding global basis functions, which are dual to the degrees of freedom $\{\texttt{G}^{\alpha}\}_{\alpha \in \mathbb{T}_k^d}$ in~\eqref{eq:globalDoF} using global normal frames, are given by $\{ T_{\beta,\alpha}\psi_{\alpha}\}$, where $(T_{\alpha, \beta})$ is the transformation matrix defined in~\eqref{eq:transformmatrix} and $(T_{ \beta, \alpha})$ is its transpose.
\end{theorem}

\section{Numerical Results}\label{sec:numerics}
In this section, we shall numerically demonstrate the effectiveness of the
proposed method.
The code is implemented by using the FEALPy package \cite{fealpy}.

\subsection{Interpolation of Smooth Functions}
The use of $\texttt{b}^\alpha$ allows for an easy construction of basis functions, but it cannot act on general smooth functions. Therefore, a classical nodal interpolation operator cannot be defined.  
To address this, we extend the DoFs in \eqref{eq:bDoF} for polynomials to $(C^m(T))'$ as follows:
\begin{equation}\label{eq:globalDoFInt}
    \langle \texttt{b}^{\alpha_f}, \Pi_{|\alpha_f|}^f\frac{\partial^{|\alpha_{f^*}|}}{\partial n_f^{\alpha_{f^*}}}u \mid_f \rangle, \quad \alpha = \alpha_f + \alpha_{f^*} \in S_{\ell}(f).
\end{equation}
Here, $\Pi_{k}^f$ denotes the Lagrange interpolation operator of degree $k$.
It is easy to see that if $u$ is a polynomial of degree $k$ on $T$, then 
DoFs \eqref{eq:bDoF} equal to the extended DoFs
\eqref{eq:globalDoFInt}.  
We define the local interpolation operator $I_T^k : C^m(T) \to \mathbb{P}_k(T)$
such that $I_T^k(u)$ has the same DoFs \eqref{eq:globalDoFInt} as $u$.

Similarly, the global DoFs \eqref{eq:globalDoF} can be extended to $(C^m(\Omega))'$ as follows:
\begin{equation}
    \label{eq:DoFglobalonmeshInt}
    \langle \texttt{b}^{\alpha_f}, \Pi_{k-|\gamma|}^f\frac{\partial^{|\gamma|}}{\partial N_f^{\gamma}}u \rangle, \quad (\alpha_f, \gamma) \in \hat{S}_{\ell}([f]).
\end{equation}
The global interpolation operator $I_h^k : C^m(\Omega) \to \mathcal S_{k,\bs r}^d(\mathcal T_h)$ is defined so that $I_h^k(u)$ has the same DoFs \eqref{eq:DoFglobalonmeshInt} as $u$. Following the convention of finite element, we denote by $u_I = I_h^k(u)$.

Although different normal derivatives are used in the DoFs \eqref{eq:globalDoFInt} and \eqref{eq:DoFglobalonmeshInt}, when $u\in C^m(\Omega)$, the local interpolation $I_T^k(u)$ will be consistent with the global interpolation $I_h^k(u)$. That is, $I_T^k(u)$, defined piecewise, will give a smooth finite element function $I_h^k(u)$ as the transformation of normal derivatives holds for smooth functions to be interpolated. So for interpolation, local DoFs and local basis are enough. 

We present several numerical examples to verify the convergence of the interpolation for smooth functions. Our focus is on the case \( m \geq 1 \), since for \( m = 0 \) the Lagrange element is already well known. The results in Tables~\ref{tab:2d2}-\ref{tab:3d2} show that the convergence rates are optimal.


%
%

\begin{table}[H]
  \centering
  \caption{Interpolation error of $u =\sin(4x)\cos(5y)$ with $\Omega = (0,1)^2$ and convergence rate for $k=7$, $m=1$, and $d=2$.}
  \renewcommand{\arraystretch}{1.125}
  \begin{tabular}{@{}c c c c c c c c@{}}
    \toprule
    $h$ &\#DoF
    & $\|u - u_I\|$ & Rate 
    & $\|\nabla u - \nabla u_I\|$ & Rate 
    & $\|\nabla^2 u - \nabla^2 u_I\|$ & Rate \\
    \midrule
    $1$   & 55   & 1.13e-01 & --   & 6.97e-01 & --   & 6.92e+00 & --   \\
    $1/2$ & 158  & 6.33e-04 & 7.48 & 7.86e-03 & 6.47 & 1.58e-01 & 5.45 \\
    $1/4$ & 526  & 2.52e-06 & 7.97 & 6.23e-05 & 6.98 & 2.50e-03 & 5.98 \\
    $1/8$ & 1910  & 1.00e-08 & 7.97 & 4.96e-07 & 6.97 & 3.99e-05 & 5.97 \\
    \bottomrule
  \end{tabular}
  \label{tab:2d2}
\end{table}

\begin{table}[H]
  \centering
  \caption{Interpolation error of $u =\sin(4x)\cos(5y)$  with $\Omega = (0,1)^2$ and convergence rate for $k=9$, $m=2$, and $d=2$.}
  \renewcommand{\arraystretch}{1.15}
  \begin{tabular}{c c 
                  c c 
                  c c 
                  c c 
                  c c}
    \toprule
    $h$ &\#DoF
    & $\|u - u_I\|$ & Rate
    & $\|\nabla u - \nabla u_I\|$ & Rate
    & $\|\nabla^2 u - \nabla^2 u_I\|$ & Rate
    & $\|\nabla^3 u - \nabla^3 u_I\|$ & Rate \\
    \midrule
    $1$   & 77  & 6.85e-02 & --   & 4.02e-01 & --   & 3.14e+00 & --   & 4.18e+01 & --   \\
    $1/2$ & 191  & 1.03e-04 & 9.38 & 1.20e-03 & 8.39 & 1.87e-02 & 9.39 & 5.01e-01 & 6.38 \\
    $1/4$ & 575  & 1.05e-07 & 9.94 & 2.44e-06 & 8.95 & 7.58e-05 & 7.95 & 4.06e-03 & 6.95 \\
    $1/8$ & 1967  & 1.05e-10 & 9.96 & 4.90e-09 & 8.96 & 3.04e-07 & 7.96 & 3.26e-05 & 6.96 \\
    \bottomrule
  \end{tabular}
  \label{tab:2d3}
\end{table}

\begin{table}[H]
  \centering
  \caption{Interpolation error of $u =\sin(2\pi x)\sin(2\pi y)\sin(2\pi z)$ with $\Omega = (0,1)^3$ and convergence rate for $k=11$, $m=1$, and $d=3$.}
  \renewcommand{\arraystretch}{1.125}
  \begin{tabular}{@{}c c c c c c c c@{}}
    \toprule
    $h$ &\#DoF
    & $\|u - u_I\|$ & Rate 
    & $\|\nabla u - \nabla u_I\|$ & Rate 
    & $\|\nabla^2 u - \nabla^2 u_I\|$ & Rate \\
    \midrule
    1    & 1158   & 1.88e-01 & --    & 1.76e+01 & --    & 1.43e+02 & --    \\
    1/2  & 6385   & 4.00e-03 & 5.55  & 4.53e-02 & 8.60  & 8.02e-01 & 7.48  \\
    1/4  & 42279  & 1.06e-06 & 11.88 & 2.46e-05 & 10.88 & 8.56e-04 & 9.87  \\
    1/8  & 307723 & 2.88e-10 & 11.85 & 1.32e-08 & 10.87 & 9.18e-07 & 9.86  \\
    \bottomrule
  \end{tabular}
  \label{tab:3d2}
\end{table}

\subsection{Conforming Finite Element Methods for Polyharmonic Equations}
The polyharmonic equation of order $m\in \mathbb N$ is given by
$$
\left\{
\begin{aligned}
    (-1)^{m+1}\Delta^{m+1}u & = f \quad\;\; \text{in}\ \Omega,\\
    u & = g_0 \quad\; \text{on}\ \partial\Omega,\\
\frac{\partial^k u}{\partial n^k} & = g_k \ \quad \text{for}\ k = 1, 2,\ldots,m
\quad  \text{on}\ \partial\Omega,
\end{aligned}
\right.
$$ 
where $\Omega$ is a polyhedral domain in $\mathbb{R}^d$, $n$ represents the outward normal vector on the boundary $\partial \Omega$.
The polyharmonic equation generalizes the Poisson equation ($m+1 = 1$) and the biharmonic equation ($m+1 = 2$).  
The variational formulation is: find $u \in H^m(\Omega)$ with trace $\tr u = (g_0, g_1, \ldots, g_m)$ such that
$$
(\nabla^{m+1} u, \nabla^{m+1} v) = (f, v), \quad \forall v \in H_0^{m+1}(\Omega).
$$
The boundary data $(g_0, g_1, \ldots, g_m)$ must satisfy certain compatibility conditions to ensure existence of a solution. Care is needed when imposing boundary conditions for smooth finite element. For details on boundary treatment, we refer to the documentation of FEALPy~\cite{fealpy}.




While strongly imposing boundary conditions can be difficult, the Nitsche technique \cite{Nitsche1971} offers a method for weak imposition. However, this approach often results in a more complex discrete bilinear form.


\begin{table}[H]
  \centering
  \caption{Finite element error and convergence rate for $k=5$, $m=1$, and $d=2$ for the biharmonic equation with $\Omega=(0, 1)^2,$ $u=(\sin(2\pi x)\sin(2\pi y))^2$, and zero Dirichlet boundary condition.}
  \renewcommand{\arraystretch}{1.125}
  \begin{tabular}{@{}c c c c c c c c@{}}
    \toprule
    $h$ &\#DoF 
    & $\|u - u_h\|$ & Rate 
    & $\|\nabla u - \nabla u_h\|$ & Rate 
    & $\|\nabla^2 u - \nabla^2 u_h\|$ & Rate \\
    \midrule
    $1/4$  & 206   & 1.61e-02 & --   & 3.37e-01 & --   & 8.31e+00 & --   \\
    $1/8$  & 694   & 3.00e-04 & 5.74 & 1.42e-02 & 4.57 & 7.54e-01 & 3.46 \\
    $1/16$ & 2534  & 3.08e-06 & 6.61 & 3.23e-04 & 5.45 & 4.28e-02 & 4.14 \\
    $1/32$ & 9670  & 3.31e-08 & 6.54 & 7.74e-06 & 5.39 & 2.34e-03 & 4.19 \\
    $1/64$ & 37766 & 4.42e-10 & 6.23 & 2.15e-07 & 5.17 & 1.39e-04 & 4.08 \\
    \bottomrule
  \end{tabular}
  \label{tab:error_rates_k5_m1}
\end{table}


\begin{table}[H]
  \centering
  \caption{Finite element error and convergence rate for $k=9$, $m=2$, and $d=2$ for the triple harmonic equation with $\Omega=(0,1)^2$, $u=\sin(2\pi x)\sin(2\pi y)$, and non-homogeneous Dirichlet boundary condition.}
  \renewcommand{\arraystretch}{1.15}
  \begin{tabular}{c c 
                  c c 
                  c c 
                  c c 
                  c c}
    \toprule
    $h$ &\#DoF
    & $\|u - u_h\|$ & Rate 
    & $\|\nabla u - \nabla u_h\|$ & Rate 
    & $\|\nabla^2 u - \nabla^2 u_h\|$ & Rate 
    & $\|\nabla^3 u - \nabla^3 u_h\|$ & Rate \\
    \midrule
    $1$   & 77     & 1.07e+00 & --    & 3.05e+00 & --   & 4.59e+02 & --   & 1.71e+02 & --   \\
    $1/2$ & 191    & 4.85e-04 & 11.11 & 7.03e-03 & 8.76 & 1.55e+00 & 8.21 & 2.73e+00 & 5.97 \\
    $1/4$ & 575    & 4.51e-07 & 10.07 & 1.21e-05 & 9.18 & 5.53e-04 & 8.13 & 2.12e-02 & 7.01 \\
    $1/8$ & 1967    & 3.74e-10 & 10.24 & 2.17e-08 & 9.12 & 1.74e-06 & 8.31 & 1.39e-04 & 7.16 \\
    \bottomrule
  \end{tabular}
  \label{tab:example3}
\end{table}

We give a numerical example in two dimensions for the biharmonic equation ($k = 5$) and the triple harmonic equation ($k = 9$). 
The numerical results are shown in Table~\ref{tab:error_rates_k5_m1} and Table~\ref{tab:example3}, which demonstrate optimal convergence rates.

We also test the biharmonic equation in three dimensions in Table~\ref{tab:error_rates_k9_m1}. The lowest degree of
polynomial for $m=1$ is $k =
9$~\cite{vzenivsek1974tetrahedral,Lai;Schumaker:2007Trivariate,zhang_family_2009}.
For $m = 2$, the matrix is very ill-conditioned even for a coarse mesh, e.g.
$h=1/2$. In the future, we will study the fast solvers for smooth elements. 

\begin{table}[H]
  \centering
  \caption{Finite element error and convergence rate for $k=9$, $m=1$, and $d=3$ for the biharmonic equation with $\Omega=(0,1)^3$, $u=\sin(5x)\sin(5y)\sin(5z)$, and non-homogeneous Dirichlet boundary condition.}
  \renewcommand{\arraystretch}{1.125}
  \begin{tabular}{@{}c c c c c c c c@{}}
    \toprule
    $h$ &\#DoF
    & $\|u - u_h\|$ & Rate 
    & $\|\nabla u - \nabla u_h\|$ & Rate 
    & $\|\nabla^2 u - \nabla^2 u_h\|$ & Rate \\
    \midrule
    $1$   & 582    & 7.35e-01 & --    & 3.51e+00 & --    & 3.01e+01 & --   \\
    $1/2$ & 2761   & 2.80e-04 & 11.36 & 5.61e-03 & 9.29  & 1.07e-01 & 8.13 \\
    $1/4$ & 16791   & 5.87e-07 & 8.90  & 2.27e-05 & 7.95  & 8.68e-04 & 6.95 \\
    $1/8$ & 116971  & 5.61e-10 & 10.03 & 4.50e-08 & 8.98  & 3.35e-06 & 8.02 \\
    \bottomrule
  \end{tabular}
  \label{tab:error_rates_k9_m1}
\end{table}

\appendix
\section{An Example of $C^1$ Finite Element in Two Dimensions}\label{appdx:C12d}
In this section, we illustrate the construction of basis functions for the lowest-order $C^1$ finite element on triangulation $\mathcal T_h$ using the procedure proposed in this paper. The smoothness vector is $\bs r = (2, 1, 0)$ and the polynomial degree is $k = 5$. The resulting space is $\mathcal{S}_{5, \bs r}^2(\mathcal{T}_h)$.

\step{1} {\bf Lattice decompositions and the local-to-global DoFs mapping}. We first construct the lattice decomposition $S_\ell(f(\texttt{T}))$ and the reference lattice decomposition $\hat{S}_\ell([f(\texttt{T})])$.  
See Fig.~\ref{fig:lattice2D} (left) for the decomposition within a triangle.  
For example, the local lattice point of $S_1([0, 1])$ is $\{ (2, 2, 1) \}$.  
We also define both local and global labeling of the lattice points, which provides the local-to-global DoFs mapping.  Details of the labeling on simplicies of different dimensions can be found in~\cite{ChenChenHuangWei2023}.

\begin{figure}[htpb]
    \centering
    \begin{minipage}{0.4\textwidth}
        \centering
        \includegraphics[width=0.65\textwidth]{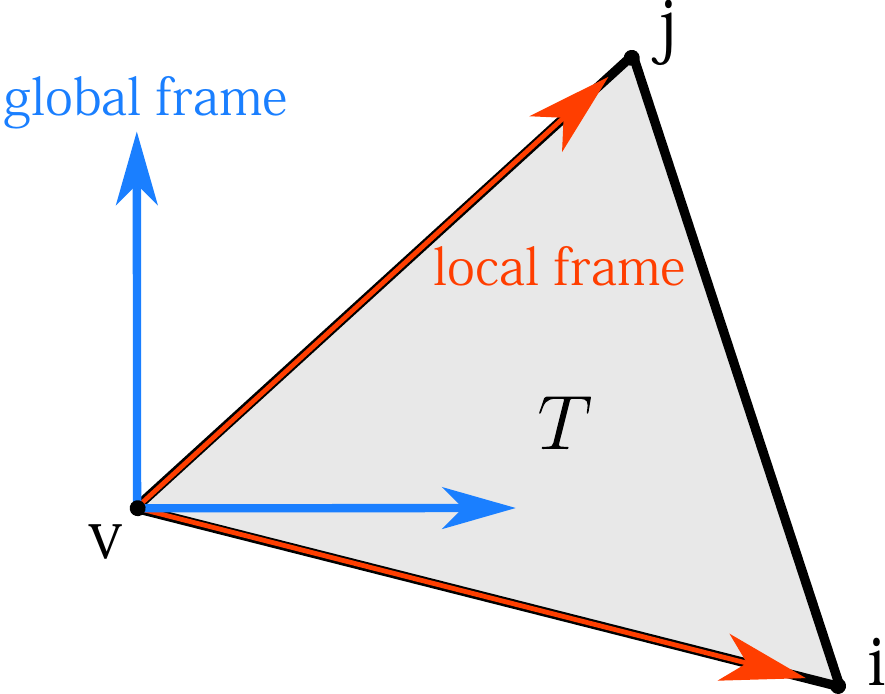}
    \end{minipage}
    \begin{minipage}{0.4\textwidth}
        \centering
        \includegraphics[width=0.72\textwidth]{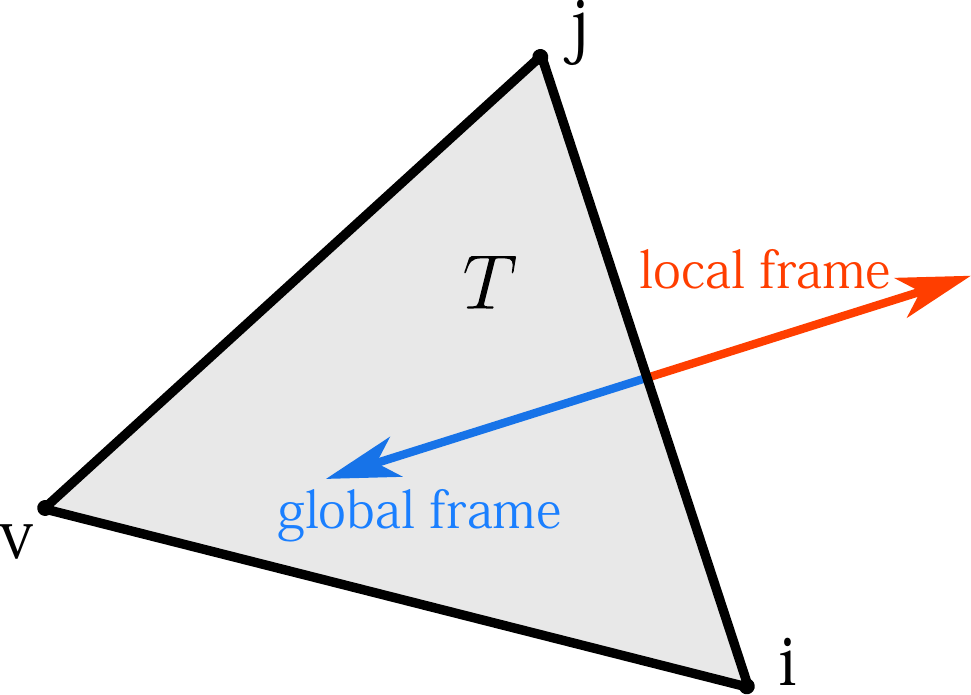}
    \end{minipage}
    \caption{The frame on vertex $\ttt{v}$(left) and edge $\{i, j\}$ (right)
    of a triangle $T$.}
    \label{fig:frame}
\end{figure}

\step{2} {\bf The local and global frames for sub-simplices}. Consider a triangular mesh $\mathcal{T}_h$, and let $T \in \mathcal{T}_h$ be a triangle with vertices $\{\bs x_0, \bs x_1, \bs x_2\}$.  
For a vertex $\texttt{v}$ of $T$, let $e_{\texttt{v}} := \{\texttt{v}\}^* = [i, j]$ denote the edge opposite to vertex $\texttt{v}$.  
We define the local normal frame $n_{\texttt{v}}$ and the global normal frame $N_{\texttt{v}}$ at vertex $\texttt{v}$ as
$$
n_{\texttt{v}} = \left\{ \bs t_{\texttt{v}, i}, \bs t_{\texttt{v}, j} \right\}, \quad
N_{\texttt{v}} = \left\{ \bs e_0, \bs e_1 \right\},
$$
where $\bs t_{\texttt{v}, i} = \bs x_i - \bs x_{\texttt{v}}$, and $\left\{ \bs e_0, \bs e_1 \right\}$ are the standard basis vectors of $\mathbb{R}^2$.

The local normal vector $n_e$ on edge $e$, which is opposite to vertex $\bs{x}_i$, is defined as
$n_e = \nabla \lambda_i/|\nabla \lambda_i|^2$. The global normal frame $N_e$ is a fixed unit normal direction of $e$, independent of the triangle containing $e$, and also determines the orientation of $e$.  
Therefore, $n_e = \sign(T, e) \, N_e/|\nabla \lambda_i|$, where $\sign(T, e) =  (N_e\cdot\nabla \lambda_i)/|\nabla \lambda_i|$ denotes the sign of the orientation of $T$ relative to $e$.  
See Figure~\ref{fig:frame} for an illustration.

\smallskip
\step{3} {\bf Local DoFs and dual basis functions}. 
The shape function space is $\mathbb{P}_5(T)$. The six local vertex degrees of freedom (DoFs) can be written as:
$$
\left\{
\begin{aligned}
u(\texttt{v}),\  
\frac{\partial u}{\partial \bs t_{\texttt{v}, i}}(\texttt{v}),\  
\frac{\partial u}{\partial \bs t_{\texttt{v}, j}}(\texttt{v}),\\ 
\frac{\partial^2 u}{\partial \bs t_{\texttt{v}, i}^2}(\texttt{v}),\  
\frac{\partial^2 u}{\partial \bs t_{\texttt{v}, i} \partial \bs t_{\texttt{v}, j}}(\texttt{v}),\  
\frac{\partial^2 u}{\partial \bs t_{\texttt{v}, j}^2}(\texttt{v})
\end{aligned}, 
\quad \texttt{v} \in \Delta_0(T),\ \{i, j\} = \texttt{v}^{*} \right.
$$
The edge DoFs are defined as:
$$
\left\langle \ttt{b}^{(2, 2)}, \frac{\partial u}{\partial n_e} \big|_e
\right\rangle, \quad e \in \Delta_1(T), \, \alpha_e = (2,2)\in \mathbb T_5^1(e).
$$

According to Section~\ref{sec:modifiedDoF}, the DoF-Basis matrix is a lower
triangular matrix. The coefficients of the basis functions using the Bernstein
polynomial basis can be obtained by inverting the transpose of this matrix, which is upper triangular. 

For this example, as $S_2(T)$ is empty, the basis
functions on the edges are computed first, followed by those at the vertices. As $|S_1(f) = 1|$, the dual basis is just a scaling of $B^{\alpha}$. For the
vertex-associated basis functions, the higher-order derivative degrees of
freedom are computed first, followed by the lower-order ones, which are sorted by the distance to the vertex. The basis functions computed later may depend on those computed earlier. 

For computational convenience, we define a matrix $\bs \Lambda_{3\times 3}$ by
$$
\Lambda_{ij} = \nabla \lambda_i \cdot n_{e_j}=\frac{\nabla \lambda_i\cdot \nabla \lambda_j}{|\nabla \lambda_j|^2}, \quad i, j = 0, 1, 2.
$$
The basis function corresponding to an
edge DoF is given by:
$$
\tilde{\phi}^{e} = 6\lambda_i^2 \lambda_j^2 \lambda_{\ttt{v}}, \quad \text{for } e = \{i, j\} 
\in \Delta_1(T),\ \texttt{v} = e^{*}.
$$
The basis functions corresponding to the vertex DoFs are:
$$
\begin{aligned}
\tilde{\phi}_3^\ttt{v} &= \frac{1}{2} \lambda_\ttt{v}^3 \lambda_i^2 - 
\frac{1}{4} \Lambda_{\ttt{v}j} \tilde{\phi}^{e_j}, \\
\tilde{\phi}_4^\ttt{v} &= \lambda_\ttt{v}^3 \lambda_i \lambda_j, \\ 
\tilde{\phi}_5^\ttt{v} &= \frac{1}{2} \lambda_\ttt{v}^3 \lambda_j^2 - 
\frac{1}{4} \Lambda_{\ttt{v}i} \tilde{\phi}^{e_i}, \\
\tilde{\phi}_1^\ttt{v} &= \lambda_\ttt{v}^4 \lambda_i + 8 \tilde{\phi}_3^\ttt{v}
+ 4 \tilde{\phi}_4^\ttt{v}, \\
\tilde{\phi}_2^\ttt{v} &= \lambda_\ttt{v}^4 \lambda_j + 8 \tilde{\phi}_5^\ttt{v}
+ 4 \tilde{\phi}_4^\ttt{v}, \\
\tilde{\phi}_0^\ttt{v} &= \lambda_\ttt{v}^5 - 20 \tilde{\phi}_3^\ttt{v} - 20
\tilde{\phi}_4^\ttt{v} - 20 \tilde{\phi}_5^\ttt{v} + 5 \tilde{\phi}_1^\ttt{v} +
5 \tilde{\phi}_2^\ttt{v}.
\end{aligned}
$$

\step{4} {\bf Global DoFs and dual basis functions}. 
We now consider the construction of global basis functions. According to
Section~\ref{sec:local2global}, global basis functions are constructed from
local basis functions via transformation matrices.

The global basis function associated with edge $e$ is given by:
$$
\phi^e = \sign(T, e) \frac{1}{|\nabla \lambda_i|}\tilde{\phi}^e.
$$
For a symmetric matrix $\bs S_{2\times 2}$,
we define
$
\mathrm{vec}(\bs S) = (S_{00}, 2S_{01}, S_{11})^{\intercal}.
$
We define the transformation matrices:
$$
\begin{aligned}
\bs C^{\texttt{v},2} &= \begin{pmatrix}
\mathrm{vec}(\bs t_{\texttt{v},i}\otimes \bs t_{\texttt{v}, i}) &
\mathrm{vec}(\mathrm{sym}(\bs t_{\texttt{v},i}\otimes \bs t_{\texttt{v}, j})) &
\mathrm{vec}(\bs t_{\texttt{v}, j}\otimes \bs t_{\texttt{v}, j})
\end{pmatrix}_{3\times 3},\\
\bs C^{\texttt{v},1} &=
\begin{pmatrix}
\bs t_{\texttt{v},i} & \bs t_{\texttt{v},j}
\end{pmatrix}_{2\times 2}.
\end{aligned}
$$
The global basis functions associated with vertex $\texttt{v}$ are defined as:
$$
\begin{aligned}
(\phi^\ttt{v}_3, \phi^\ttt{v}_4, \phi^\ttt{v}_5)^\intercal &=
\bs C^{\ttt{v},2}(\tilde{\phi}^\ttt{v}_3, \tilde{\phi}^\ttt{v}_4, \tilde{\phi}^\ttt{v}_5)^{\intercal},\\
(\phi^\ttt{v}_1, \phi^\ttt{v}_2)^{\intercal} &=
\bs C^{\ttt{v},1}(\tilde{\phi}^\ttt{v}_1, \tilde{\phi}^\ttt{v}_2)^{\intercal},
\\
\phi^\ttt{v}_0 &= \tilde{\phi}^\ttt{v}_0.
\end{aligned}
$$
Together, these form a basis for the lowest-order $C^1$ finite element space $\mathcal S_{5, \bs r}^2(\mathcal T_h)$ on a triangulation $\mathcal T_h$.

It is worth noting that the Argyris element \cite{argyris_tuba_1968} shares the same shape function space and smoothness vector, but its edge DoFs are defined differently, as:
$$
\texttt{m}_e(u):= \frac{\partial u}{\partial n_e} (\bs m_e), \quad e \in \Delta_1(T),
$$
where $\bs m_e$ denotes the midpoint of edge $e$. Basis functions dual to the corresponding DoFs
can be constructed similarly. The edge basis functions differ from
$\phi^e$ only by a scalar coefficient. However, the computation of the vertex basis functions becomes more complicated as $\texttt{m}_e(B^{\alpha})$ has no simple formulae. The modification using $\texttt{b}^{\alpha}$ as DoFs simplifies the construction and generalizes to all cases. 
We refer to \cite{Okabe1993,DominguezSayas2009} for the basis functions for the Argyris element.

\section{Notation Table}

In this section, we list the notation used throughout the paper in the following table for easy reference.
{\renewcommand{\arraystretch}{1.35} 
\small
\begin{longtable}{>{\centering\arraybackslash}m{0.1254\textwidth} m{0.85\textwidth}}
\caption{Notation list.} \\
    \toprule

\textbf{Notation} & \textbf{Description} \\
\endfirsthead
\hline

\textbf{Notation} & \textbf{Description} \\
\endhead
\hline

$T$ & Geometric $d$-simplex with vertices $\{\texttt{v}_0, \ldots, \texttt{v}_d\}$ \\
$\hat{T}$ & Standard reference simplex with vertices $\{0, \bs e_1, \ldots, \bs e_d\}$ \\
$\TT$ & Abstract $d$-simplex, i.e., a finite set of cardinality $d+1$ \\
$\texttt{S}_d$ & Standard (combinatorial) $d$-simplex $\{0, 1, \ldots, d\}$ \\
$\mathcal{T}_h$ & A geometric triangulation of a domain $\Omega$ \\
$\mathcal{S}$ & A simplicial complex \\
$\Delta_{\ell}(T)$ & Set of all $\ell$-dimensional sub-simplices ($\ell$-simplicies) of $T$ \\
$\Delta_{\ell}(\mathcal{S})$ & Set of all $\ell$-simplicies in simplicial complex $\mathcal{S}$ \\
$\Delta_{\ell}(\mathcal{T}_h)$ & Set of all $\ell$-simplicies in triangulation $\mathcal{T}_h$ \\
\\
$\boldsymbol{r}$ & Smoothness vector: a sequence of integers $(r_0, \ldots, r_d)$ satisfying $r_d=0$, $r_{d-1}=m$, $r_{\ell} \ge 2r_{\ell+1}$ \\
$\mathbb{T}_k^d$ & Simplicial lattice of degree $k$ in dimension $d$: $\left\{ \alpha \in \mathbb{N}^{0:d} : |\alpha| = k \right\}$ \\
$\mathbb{T}_k^d(D)$ & Set of lattice points whose geometric images lie within $D \subseteq T$ \\
$\alpha_f$ & Restriction of multi-index $\alpha$ to sub-simplex $f$, $R_f(\alpha)$ \\
$\alpha!$ & Factorial of a multi-index, $\alpha_0! \alpha_1! \cdots \alpha_d!$ \\
$ \dist(\alpha, f)$ & distance of lattice point $\alpha$ to sub-simplex $f$: $|\alpha_{f*}| = \sum_{i\in f^*}\alpha_i$\\
$L(f,s)$ &  $\left\{ \alpha \in \mathbb{T}_k^d : \dist(\alpha, f) = s \right\}$, $L(f, s) = L(f^*, k - s)$ \\
$D(f, r)$ & $\left\{ \alpha \in \mathbb{T}_k^d : \dist(\alpha, f) \le r \right\}$, $D(f, r) = \bigcup_{s=0}^r L(f, s)$\\
$S_{\ell}(f)$ & $D(f, r_\ell) \setminus \left( \bigcup_{i=0}^{\ell-1} \bigcup_{e \in \Delta_i(f)} D(e, r_i) \right)$\\
$\hat{S}_{\ell}([f])$ & Reference lattice set $\{(\alpha_f, \gamma) \mid \alpha_f \in R_f(S_{\ell}(f)), \gamma \in \mathbb{T}_{k - |\alpha_{f}|}^{d-\ell - 1} \}$ \\
$\mathbb{T}_k^d$ & Lattice decomposition $\Oplus_{\ell=0}^d \Oplus_{f \in \Delta_\ell(\TT)} S_\ell(f)$ \\
$\mathbb S_{k,\bs r}^d$ & Lattice decomposition on triangulation $\mathcal T_h$ with smoothness vector $\bs r$ $\Oplus_{\ell=0}^d\Oplus_{f\in \Delta_{\ell}(\mathcal T_h)}\hat{S}_{\ell}([f])$\\
\\
$\lambda$ & Barycentric coordinates $(\lambda_0, \lambda_1, \ldots, \lambda_d)$ \\
$\lambda^\alpha$ & Monomial in barycentric coordinates, $\lambda_0^{\alpha_0} \cdots \lambda_d^{\alpha_d}$ for $\alpha\in \mathbb N^{0:d}$ \\
$\lambda_f $ & Barycentric coordinates $(\lambda_{f(0)}, \ldots, \lambda_{f(\ell)})$ associated with the vertices of $f\in \Delta_{\ell}(T)$ \\
$\lambda_f^{\alpha}$ & $\lambda_{f(0)}^{\alpha_0} \cdots \lambda_{f(\ell)}^{\alpha_{\ell}}$ for $f\in \Delta_{\ell}(T)$ and $\alpha\in \mathbb N^{0:\ell}$ \\
$\mathbb{P}_k(T)$ & Space of polynomials of degree at most $k$ on simplex $T$ \\
$\mathcal{B}_k$ & Bernstein basis for $\mathbb{P}_k(T)$: $\left\{ B^{\alpha}= \frac{k!}{\alpha!}\lambda^{\alpha}: \alpha \in \mathbb{T}_k^d \right\}$ \\
$\mathcal{B}_k(f)$ & Bernstein basis over sub-simplex $f$, $\{ B^{\alpha}_f := \frac{k!}{\alpha!} \lambda_f^{\alpha} : \alpha \in \mathbb{T}_k^\ell\}$ \\
$\{\texttt{b}^\alpha\}$ & Dual basis of $\mathbb{P}_k(T)'$ to the Bernstein basis $\{B^\beta\}$: $\langle \texttt{b}^{\alpha}, B^{\beta} \rangle = \delta_{\alpha, \beta}$ \\
\\
$\mathbb{R}^{d,r}$ & $(\mathbb{R}^d)^{\otimes r}$: $r$-th order tensor space over $\mathbb{R}^d$  \\
$\mathbb{S}^{d,r}$ & $r$-th order symmetric tensor space over $\mathbb{R}^d$: $\tau_{i_{\sigma(1)} \cdots i_{\sigma(r)}} = \tau_{i_1 \cdots i_r}$  for any $\sigma \in \mathcal{G}^r$ \\
$\sym(\boldsymbol{\tau})$ & Symmetrization of a tensor $\boldsymbol{\tau}$: ${\rm sym}(\boldsymbol{\tau})_{i_1 \cdots i_r} = \frac{1}{r!} \sum_{\sigma \in \mathcal{G}^r} \tau_{i_{\sigma(1)} \cdots i_{\sigma(r)}}, \quad 1 \leq i_1, \ldots, i_r \leq d.
$ \\
$\bs t_i^{\otimes\alpha_i}$ & Tensor product $\bs t_i \otimes \cdots \otimes \bs t_i$ ($\alpha_i$ times) \\
$t^\alpha$ or $t^{\otimes\alpha}$ & Tensor monomial $\bs t_1^{\otimes\alpha_1} \otimes \cdots \otimes \bs t_\ell^{\otimes\alpha_\ell}$ \\
$\mathcal{I}_{\ell}^r$ & Increasing multi-index set $\{(i_1, \ldots, i_r) \in \{1, \ldots, \ell\}^r : i_1 \le \cdots \le i_r\}$ \\
$I^\alpha$ & Increasing multi-index in $\mathcal{I}_{\ell}^r$ corresponding to $\alpha \in \mathbb{T}_r^{\ell-1}$ \\
$\mathcal{G}^r$ & Permutation group of $(1, \ldots, r)$ \\
$\mathcal{G}_\alpha^r$ & Set of equivalence classes $\mathcal{G}^r / \sim^\alpha$: $\sigma \sim^\alpha \sigma' \iff \sigma(I^\alpha) = \sigma'(I^\alpha)$ \\
$\frac{\partial^rv}{\partial n^\alpha} $ & Directional derivative of order $r$ of a function $v$ with respect to directions $n = \{\boldsymbol{n}_1, \dots, \boldsymbol{n}_\ell\}$ and multi-index $\alpha \in \mathbb{T}_r^{\ell-1}$: $\nabla^r v : n^{\otimes \alpha}$  \\
$\partial^\alpha v$ & Partial derivative $\frac{\partial^r v}{\partial x_1^{\alpha_1} \cdots \partial x_d^{\alpha_d}}$ \\
\\
$\mathscr{T}^f$ & Tangent plane of sub-simplex $f$ \\
$\mathscr{N}^f$ & Normal plane of sub-simplex $f$ \\
$\nabla \lambda$ & Vector of gradients of barycentric coordinates $(\nabla\lambda_0, \ldots, \nabla\lambda_d)$ \\
$\nabla_f \lambda_i$ & Surface or tangential gradient of $\lambda_i$ on $f$ \\
$n_f$ & Local normal basis of $\mathscr N^f$, $(\bs{n}_f^1, \ldots, \bs{n}_f^{d-\ell}) = \left\{ \frac{\nabla_{f \cup \{i\}} \lambda_i}{|\nabla_{f \cup \{i\}} \lambda_i|^2} \mid i \in f^* \right\}$ \\
$\hat n_f$ & Dual normal basis of $n_f$: $(\hat{\bs{n}}_f^1, \ldots, \hat{\bs{n}}_f^{d-\ell}) = \left\{ \nabla \lambda_i \mid i \in f^* \right\}$\\
$N_f$ & Global basis $\{\bs{N}_f^1, \ldots, \bs{N}_f^{d-\ell}\}$ for the normal plane $\mathscr{N}^f$ \\
$\hat{N}_f$ & Dual normal basis of $N_f$: $\{\hat{\boldsymbol{N}}_f^1, \ldots, \hat{\boldsymbol{N}}_f^{d-\ell}\}$ \\
\\
$\texttt{L}_f^\alpha(u)$ & Modified local degrees of freedom $\langle \texttt{b}^{\alpha_f}, \frac{\partial^{|\alpha_{f^*}|}}{\partial n_f^{\alpha_{f^*}}} u \mid _f \rangle$\\
$\texttt{G}_\alpha(u)$ & Global degrees of freedom $\langle \texttt{b}^{\alpha_f}, \frac{\partial^{|\gamma|}}{\partial N_f^{\gamma}}u \mid_f \rangle$\\
$D_{\alpha, \beta}$ & Local DoF-Basis matrix $\texttt{L}_f^{\alpha}(B^{\beta}), \, \alpha, \beta \in \mathbb{T}_k^d(T)$ \\
& $D_{\alpha, \beta} = \frac{k!}{(k - s)!} \delta_{\alpha, \beta}, \quad \alpha, \beta \in S_{\ell}(f) \cap L(f, s)$\\
$T_{\alpha,\beta}$ & Transformation matrix from local to global DoFs \\
\\
$\Pi^f_k$ & Lagrange interpolation operator of degree $k$ on $f$: $\Pi^f_ku(x_{\alpha}) = u(x_{\alpha})$ for all interpolation points $x_{\alpha}$ at lattice point $\alpha$\\
$I_T^k$ & Local interpolation operator $C^m(T) \rightarrow \mathbb{P}_k(T)$: $\texttt{L}_f^\alpha(I_T^k u) = \texttt{L}_f^\alpha(u) $ \\
$I_h^k$ & Global interpolation operator $C^m(\Omega) \rightarrow S_{k,\boldsymbol{r}}^d(\mathcal{T}_h)$: $\texttt{G}_\alpha(I_h^k u) = \texttt{G}_\alpha(u)$ \\
\bottomrule

\end{longtable}
}


\begin{thebibliography}{10}

\bibitem{ainsworth2011bernstein}
M.~Ainsworth, G.~Andriamaro, and O.~Davydov.
\newblock Bernstein--{B}{\'e}zier finite elements of arbitrary order and
  optimal assembly procedures.
\newblock {\em SIAM Journal on Scientific Computing}, 33(6):3087--3109, 2011.

\bibitem{argyris_tuba_1968}
J.~H. Argyris, I.~Fried, and D.~W. Scharpf.
\newblock The {TUBA} family of plate elements for the matrix displacement
  method.
\newblock {\em Aero. J. Roy. Aero. Soc.}, 72(692):701--709, Aug. 1968.

\bibitem{ArnoldFalkWinther2009}
D.~N. Arnold, R.~S. Falk, and R.~Winther.
\newblock Geometric decompositions and local bases for spaces of finite element
  differential forms.
\newblock {\em Computer Methods in Applied Mechanics and Engineering},
  198(21-26):1660--1672, May 2009.

\bibitem{jax2018github}
J.~Bradbury, R.~Frostig, P.~Hawkins, M.~J. Johnson, C.~Leary, D.~Maclaurin,
  G.~Necula, A.~Paszke, J.~Vander{P}las, S.~Wanderman-{M}ilne, and Q.~Zhang.
\newblock {JAX}: composable transformations of {P}ython+{N}um{P}y programs,
  2018.

\bibitem{bramble_triangular_1970}
J.~H. Bramble and M.~Zl{\'a}mal.
\newblock Triangular elements in the finite element method.
\newblock {\em Math. Comp.}, 24(112):809--820, 1970.

\bibitem{ChenChenHuangWei2023}
C.~Chen, L.~Chen, X.~Huang, and H.~Wei.
\newblock Geometric decomposition and efficient implementation of high order
  face and edge elements.
\newblock {\em Commun. Comput. Phys.}, 35(4):1045--1072, 2024.

\bibitem{Chen2008ifem}
L.~Chen.
\newblock $i${FEM}: an integrated finite element methods package in {MATLAB}.
\newblock Technical report, University of California at Irvine, 2008.

\bibitem{ChenFEM}
L.~Chen.
\newblock Introduction to finite element methods.
\newblock Lecture Notes, 2008.

\bibitem{Chen2018Programming}
L.~Chen.
\newblock Programming of finite element methods in {MATLAB}.
\newblock {\em arXiv preprint arXiv:1804.05156}, 2018.

\bibitem{chen_geometric_2021}
L.~Chen and X.~Huang.
\newblock Geometric decompositions of the simplicial lattice and smooth finite
  elements in arbitrary dimension, 2021.

\bibitem{Chen;Huang:2022FEMcomplex3D}
L.~Chen and X.~Huang.
\newblock Finite element de {R}ham and {S}tokes complexes in three dimensions.
\newblock {\em Math. Comp.}, 93(345):55--110, 2024.

\bibitem{Chen;Huang:2021Geometric}
L.~Chen and X.~Huang.
\newblock {$H ({\rm div})$}-conforming finite element tensors with constraints.
\newblock {\em Results Appl. Math.}, 23:Paper No. 100494, 33, 2024.

\bibitem{Chui;Lai:1990Multivariate}
C.~K. Chui and M.-J. Lai.
\newblock Multivariate vertex splines and finite elements.
\newblock {\em Journal of approximation theory}, 60(3):245--343, 1990.

\bibitem{Ciarlet1978}
P.~G. Ciarlet.
\newblock {\em The finite element method for elliptic problems}.
\newblock North-Holland Publishing Co., Amsterdam, 1978.

\bibitem{Davydov2001}
O.~Davydov.
\newblock Stable local bases for multivariate spline spaces.
\newblock {\em J. Approx. Theory}, 111(2):267--297, 2001.

\bibitem{DominguezSayas2009}
V.~Dom\'{\i}nguez and F.-J. Sayas.
\newblock Algorithm 884: a simple {M}atlab implementation of the {A}rgyris
  element.
\newblock 35(2):Art. 16, 11.

\bibitem{Farin2002}
G.~Farin.
\newblock {\em Curves and surfaces for {CAGD}: a practical guide}.
\newblock Morgan Kaufmann, San Francisco, fifth edition edition, 2002.

\bibitem{harris2020array}
C.~R. Harris, K.~J. Millman, S.~J. Van Der~Walt, R.~Gommers, P.~Virtanen,
  D.~Cournapeau, E.~Wieser, J.~Taylor, S.~Berg, N.~J. Smith, et~al.
\newblock Array programming with numpy.
\newblock {\em Nature}, 585(7825):357--362, 2020.

\bibitem{hatcher2005algebraic}
A.~Hatcher.
\newblock {\em Algebraic topology}.
\newblock Cambridge University Press, Cambridge, 2002.

\bibitem{HuLinWu2024}
J.~Hu, T.~Lin, and Q.~Wu.
\newblock A construction of {$C^r$} conforming finite element spaces in any
  dimension.
\newblock {\em Found. Comput. Math.}, 24(6):1941--1977, 2024.

\bibitem{LaiSchumaker2007}
M.-J. Lai and L.~L. Schumaker.
\newblock {\em Spline functions on triangulations}, volume 110 of {\em
  Encyclopedia of Mathematics and its Applications}.
\newblock Cambridge University Press, Cambridge, 2007.

\bibitem{Lai;Schumaker:2007Trivariate}
M.-J. Lai and L.~L. Schumaker.
\newblock Trivariate {$C^r$} polynomial macroelements.
\newblock {\em Constr. Approx.}, 26(1):11--28, 2007.

\bibitem{nicolaides1973class}
R.~Nicolaides.
\newblock On a class of finite elements generated by {L}agrange interpolation.
  {II}.
\newblock {\em SIAM J. Numer. Anal.}, 10(1):182--189, 1973.

\bibitem{nicolaides1972class}
R.~A. Nicolaides.
\newblock On a class of finite elements generated by {L}agrange interpolation.
\newblock {\em SIAM J. Numer. Anal.}, 9:435--445, 1972.

\bibitem{Nitsche1971}
J.~Nitsche.
\newblock \"{U}ber ein {V}ariationsprinzip zur {L}\"{o}sung von
  {D}irichlet-{P}roblemen bei {V}erwendung von {T}eilr\"{a}umen, die keinen
  {R}andbedingungen unterworfen sind.
\newblock 36:9--15.

\bibitem{Okabe1993}
M.~Okabe.
\newblock Full-explicit interpolation formulas for the {A}rgyris triangle.
\newblock 106(3):381--394.

\bibitem{paszke2019pytorch}
A.~Paszke, S.~Gross, F.~Massa, A.~Lerer, J.~Bradbury, G.~Chanan, T.~Killeen,
  Z.~Lin, N.~Gimelshein, L.~Antiga, et~al.
\newblock Pytorch: An imperative style, high-performance deep learning library.
\newblock {\em Advances in neural information processing systems}, 32, 2019.

\bibitem{Wasserman2009}
R.~H. Wasserman.
\newblock {\em Tensors and manifolds}.
\newblock Oxford University Press, Oxford, second edition, 2009.
\newblock With applications to physics.

\bibitem{fealpy}
H.~Wei and Y.~Huang.
\newblock {FEALPy}: Finite element analysis library in {Python}. {GitHub} -
  weihuayi/fealpy: Finite element analysis library in python, Xiangtan
  University, 2017-2024.

\bibitem{zenisek_interpolation_1970}
A.~{\v Z}en{\'\i}{\v s}ek.
\newblock Interpolation polynomials on the triangle.
\newblock {\em Numer. Math.}, 15(4):283--296, Aug. 1970.

\bibitem{vzenivsek1974tetrahedral}
A.~{\v{Z}}en{\'\i}{\v{s}}ek.
\newblock Tetrahedral finite ${C}^m$ elements.
\newblock {\em Acta Universitatis Carolinae. Mathematica et Physica},
  15(1):189--193, 1974.

\bibitem{zhang_family_2009}
S.~Zhang.
\newblock A family of {3D} continuously differentiable finite elements on
  tetrahedral grids.
\newblock {\em Applied Numerical Mathematics}, 59(1):219--233, Jan. 2009.

\bibitem{Zhang2016a}
S.~Zhang.
\newblock A family of differentiable finite elements on simplicial grids in
  four space dimensions.
\newblock {\em Mathematica Numerica Sinica}, 38(3):309--324, 2016.

\end{thebibliography}
\end{document}